\documentclass[11pt, reqno]{amsart}
\usepackage{euscript,amscd,amsgen,amsfonts,amssymb,latexsym}

\newcommand{\eqdef}{\stackrel{\scriptscriptstyle\rm def}{=}}
\usepackage{bbm}
\usepackage{pgfplots, multicol}

\newtheorem{theorem}{Theorem}

\newtheorem{proposition}{Proposition}
\newtheorem{corollary}{Corollary}

\newtheorem{lemma}{Lemma}
\newtheorem{remark}{Remark}
\newtheorem{remarks}{Remarks}
\newtheorem{example}{Example}

\newcommand{\beha}{\begin{enumerate}}
\newcommand{\behe}{\end{enumerate}}
\renewcommand{\epsilon}{\varepsilon}

\newcommand{\R}{{\rm Rot}}

\newcommand{\Or}{\mathcal{O}}

\newcommand{\cM}{\EuScript{M}}

\newcommand{\cH}{\EuScript{H}}
\newcommand{\bR}{{\mathbb R}}

\newcommand{\bN}{{\mathbb N}}

\def\1{1\!\!1}

\def\and{\text{ and }}

                        \def\^{\tilde}

\def\Per{{\rm Per}}

\def\Fix{{\rm Fix}}
\def\Fix{{\rm Fix}}
\def\Fix{{\rm Fix}}
\def\Per{{\rm Per}}

\def\1{1\!\!1}

\def\rv{{\rm rv}}

\newtheorem*{thmA}{Theorem A}

\newtheorem*{thmB}{Theorem B}

\newtheorem*{thmC}{Theorem C}

\DeclareMathSymbol{\varnothing}{\mathord}{AMSb}{"3F}
\renewcommand{\emptyset}{\varnothing}
\title[Ground States and Rotation Sets]{Ground States and Zero-Temperature Measures at the Boundary of Rotation Sets}
\author{Tamara Kucherenko}\address{Department of Mathematics, The City College of New York, New York, NY, 10031, USA}\email{tkucherenko@ccny.cuny.edu}

\author{Christian Wolf}\address{Department of Mathematics, The City College of New York, New York, NY, 10031, USA}\email{cwolf@ccny.cuny.edu}

\thanks{This work was partially supported by grants from the PSC-CUNY (TRADA-46-282 to Tamara Kucherenko), (TRADA-45-356 to Christian Wolf) and by a grant from the Simons Foundation (\#209846 to Christian Wolf)}

\begin{document}

\begin{abstract}

We consider a continuous dynamical system $f:X\to X$ on a compact metric space $X$ equipped with an $m$-dimensional continuous
potential $\Phi=(\phi_1,\cdots,\phi_m):X\to \bR^m$. We study the set of ground states $ GS(\alpha)$ of the potential $\alpha\cdot \Phi$
as a function of the direction vector $\alpha\in S^{m-1}$. 
We show that the structure of the ground state sets is naturally related to the geometry of the generalized rotation set of $\Phi$.
In particular, for each $\alpha$ the set of rotation vectors of $ GS(\alpha)$ forms a non-empty, compact and connected subset of a face $F_\alpha(\Phi)$ of the rotation set associated with $\alpha$. Moreover, every ground state
maximizes entropy among all invariant measures with rotation vectors in $F_\alpha(\Phi)$. We further establish the occurrence of several quite unexpected phenomena. Namely, we construct for any $m\in\bN$ examples with an exposed boundary point (i.e. $F_\alpha(\Phi)$ being a singleton) without a unique ground state. Further, we establish the possibility of a line segment face $F_\alpha(\Phi)$ with a unique but non-ergodic ground state. Finally, we establish the possibility that the set of rotation vectors of
$GS(\alpha)$ is a non-trivial line segment.

\end{abstract}
\keywords{Ground states, zero temperature measures, generalized rotation sets, thermodynamic formalism, entropy, equilibrium states}
\subjclass[2010]{Primary 37D35, 37E45 Secondary 37B10, 37E45, 37L40}
\maketitle

\section{Introduction}
\subsection{Motivation}
It is a central problem in the thermodynamic formalism to describe the family of equilibrium states $\mu_t$ associated with a one-parameter family of observables $t\varphi$ of a given
potential $\varphi$ on the phase space. Here the parameter $t$ is viewed as the inverse temperature $1/T$  and consequently large values of $t$  correspond to an equilibrium $\mu_t$ at
temperature close to zero.  It is then a natural problem to classify the asymptotic behavior of such measures when temperature approaches zero and, in particular, to gather information about the  corresponding limit equilibria. These limits are of great interest since they are so-called ground states, i.e. states supported on configurations with minimal energy \cite{vEFS}.  If this limit exists (i.e. the ground state is unique) we call the corresponding limit a zero temperature measure.

In this paper we consider  deterministic discrete-time dynamical systems given by a continuous map $f:X\to X$ on a compact metric space $X$.    To capture information about the statistical properties of the dynamical system we denote by $\cM$ the set of all Borel invariant probability measures. We endow $\cM$ with the weak$^*$ topology which makes $\cM$ a compact convex topological space. Given a continuous potential $\varphi:X\to \bR$ we say $\mu\in \cM$ is an equilibrium state of the potential $\varphi$ if $\mu$ maximizes "free energy," that is
\begin{equation}\label{eq111}
h_\mu(f) +\int \varphi d\mu=\sup\left\{h_\nu(f) +\int \varphi d\nu: \nu\in \cM\right\},
\end{equation}
where $h_\nu(f)$ denotes the measure-theoretic entropy of $f$ with respect to $\nu$. We note that by the variational principle the supremum on the right-hand side of \eqref{eq111} coincides with the topological pressure of $\varphi$. To avoid making vacuous statements we shall always assume that the entropy map
$\mu\mapsto h_\mu(f)$ is upper semi-continuous which guarantees that the set of equilibrium states $ES(t\varphi)$ of the potential $t\varphi$ contains at least one invariant measure. We say a measure $\mu$ is a {\em ground state} of the potential $\varphi$ if $\mu$ is  the weak$^*$ limit of a sequence of equilibrium states $\mu_{t_n}\in ES(t_n\varphi)$  for some sequence $t_n\to \infty$. 
 It follows that every ground state $\mu$ of $\varphi$ is a maximizing measure, that is $\mu$ maximizes the integral $\int \varphi d\nu$ among all invariant probability measures (see \cite{Je2}).
In the presence of a unique ground state (i.e. the limit  $\lim_{t\to\infty}\mu_t$ exists) we call this limit the {\em zero-temperature measure} of the potential $\varphi$.

Ground states and zero-temperature measures play a fundamental role in statistical physics; yet their rigorous mathematical treatment
has just been developed during the last fifteen years.
The class of systems that have been intensively studied  are systems with strong thermodynamic properties which include subshifts of finite type and expanding maps and H\"older continuous potentials. In this setup
$t\varphi$ has a unique equilibrium state $\mu_t$ and $\mu_t$ is a Gibbs measure.    Contreras, Lopes and Thieullen \cite{CLT} established the existence of the zero-temperature limit for a generic set of H\"older potentials. Later, Bremont \cite{B} proved that for subshifts of finite type and potentials that are locally constant the zero-temperature limit exists. This result had been generalized by Leplaideur \cite{Le} who also provided a new proof of Bremont's theorem. Recently, Chazottes and Hochman constructed an example of a Lipschitz continuous potential on a one-sided shift space with two distinct ergodic ground states \cite{ChH}. One of the main open questions in this area is that given a particular class of systems, does there exist a generic set of potentials for which  the zero-temperature measure is supported on a periodic orbit. This question has been positively answered by Contreras \cite{C} in the case of expanding maps and Lipschitz continuous potentials by building up  on previous results of Moris \cite{Mor}. Finally we note that ground states and zero-temperature  measures have also been studied in the context of countable Markov shifts by  Kempton \cite{K} and
Jenkinson, Mauldin and Urbanski \cite{Je4}.

Our approach in this paper is slightly different. Namely, we consider an $m$-dimensional continuous potential $\Phi=(\phi_1,\cdots,\phi_m)$ and then study the ground states and zero-temperature limits of potentials that are linear combinations
\[
\alpha_1\cdot \phi_1+\cdots +\alpha_m\cdot \phi_m
\]
of the coordinate functions of $\Phi$. Since the corresponding ground states are accumulation points of $\mu_{t\alpha\cdot \Phi}$ it suffices to consider coefficients $(\alpha_1,...,\alpha_m)=\alpha$ on the unit sphere in $\bR^m$. We refer to these  $\alpha$ as direction vectors. We then study the set of  ground states
$GS(\alpha)$ of potentials $\alpha\cdot\Phi$ as a function of $\alpha$. We refer to $\mu\in GS(\alpha)$ as a ground state in the direction of $\alpha$.  Moreover, we are interested in their rotation vectors $\rv(\mu)=(\int \phi_1 d\mu,\cdots,\int \phi_m d\mu)$ which are particular points in the rotation set $\R(\Phi)=\rv(\cM)$.

It turns out that there is a natural connection
between the rotation vectors of the ground states and the geometry of the rotation set. Intuitively, the rotation vectors of the ground states should be located on the boundary of the corresponding rotation set.
This paper is motivated by a question of Artur Oscar Lopes who asked us if the rotation classes of vectors at the boundary of a rotation set necessarily contain a ground state.
We will now describe our results in more detail.

\subsection{Statement of the Results.} Let $f:X\to X$ be a continuous map on a compact metric space.
We assume   that $f$ has finite topological entropy and that the entropy map of $f$ is upper semi-continuous. The latter guarantees that every continuous potential $\varphi$ has at least one equilibrium state $\mu_\varphi$. Given an $m$-dimensional potential  $\Phi=(\phi_1,\cdots,\phi_m)\in C(X,\bR^m)$ we denote
by $\R(\Phi)=\{\rv(\mu): \mu\in \cM\}$ the (generalized) rotation set of $\Phi$. It follows from the definition that
the rotation set is a compact and convex subset of $\bR^m$. Rotation sets are natural generalizations
of Poincar\'e's rotation number of an orientation preserving homeomorphism on a circle that have been been recently studied by several authors (\cite{B}, \cite{GKLM}, \cite{GM}, \cite{Je}, \cite{KW1}, \cite{KW3}, \cite{MZ} and \cite{Z}).

Given a rotation set $\R(\Phi)$, there is a natural correspondence between a direction vector $\alpha$ and the associated  face $F_\alpha(\Phi)$ of $\R(\Phi)$ (see Section 3 for the precise definition). The following theorem establishes the connection
between the ground states of $\alpha\cdot \Phi$ and $\R(\Phi)$ (see Theorem \ref{thm1}  in the text).

\begin{thmA}\label{thmA}
Let $f$ be a continuous map on a compact metric space for which the entropy map is upper-semi continuous. Let $\Phi:X\to\bR^m$ be a continuous potential and let $\alpha$ be a direction vector. Then
\begin{enumerate}
\item[(a)]
If  $\mu$ is a ground state in the direction of $\alpha$ then $\rv(\mu)\in F_\alpha(\Phi)$ and $h_\mu(f)= \sup\left\{h_\nu(f): \rv(\nu)\in F_\alpha(\Phi)\right\}$.
\item[(b)]
The set $\{\rv(\mu): \mu \in GS(\alpha)\}$ is compact.  Further, if for large enough $t$ there exist unique equilibrium states $\mu_t$
of the potential $t\alpha\cdot \Phi$ and $t\mapsto \rv(\mu_t)$ is continuous then $\{\rv(\mu): \mu \in GS(\alpha)\}$ is connected.
\end{enumerate}

\end{thmA}
Theorem A connects the rotation vectors of the ground states of $\alpha\cdot \Phi$  to the geometry of the boundary of the rotation set $\R(\Phi)$. This is particularly useful for systems and potentials where we have a good understanding of the corresponding rotation sets. We note that in general the geometry of rotation sets is quite complicated. Indeed, we recently proved in \cite{KW1} that for subshifts of finite type every compact and convex set $K\subset\bR^m$ is attained as the rotation set of some $m$-dimensional potential $\Phi$. On the other hand, if $f$ is a subshift of finite type and $\Phi$ is locally constant, then by Ziemian's theorem \cite{Z} the rotation set is a polyhedron and hence its boundary is the union of finitely many faces. Moreover, by Bremont's theorem the zero-temperature limit exists for all direction vectors $\alpha$. Therefore, it follows from part (a) of Theorem A that for subshifts of finite type and locally constant potentials  the set of boundary points with a ground state in its rotation class has dimension at most $m-2$. In particular, in $\bR^2$  this set is finite.

Obviously, the statement that    $\{\rv(\mu): \mu \in GS(\alpha)\}$ is connected is only of relevance in situations where $\rv(GS(\alpha))$ contains more than one rotation vector. For example, it is known (see \cite{CLT} and \cite{Je2}) that there exists a generic subset $S$ of the Banach space of H\"older continuous potentials such that if $\alpha\cdot \Phi\in S$ then there exists a unique ground state and in particular   $\{\rv(\mu): \mu \in GS(\alpha)\}$ is a singleton. As a counter part to this result we construct in Theorem \ref{thm5} an example for which  $\{\rv(\mu): \mu \in GS(\alpha)\}$ is a non-trivial line segment.

Next, we discuss different classes of boundary points. Recall that $w\in \partial \R(\Phi)$ is an exposed point if $\{w\}$ is a face of $\R(\Phi)$.
Moreover, we say the face $\{w\}$ is smooth if it corresponds to a unique direction vector $\alpha$ (see Section 2.4 for the exact definitions). It is an immediate consequence of Theorem A
that the rotation class of an exposed point $w$ contains at least one ground state $\mu$ (that is $\rv(\mu)=w$) and that $\mu$ maximizes entropy within this rotation class.
One might suspect that in the case of a smooth exposed point $w$ the zero-temperature limit necessarily exists. However, the following theorem shows that this is not true independently of the dimension $m$ (see Theorem \ref{thm3} and Corollary \ref{cor22} in the text).

\begin{thmB}\label{thmB}
Let $f:X\to X$ be the one-sided full shift map. Then for every $m\in \bN$ there exists a continuous potential $\Phi:X\to\bR^m$ such that  $\R(\Phi)$ has a smooth exposed point with direction vector $\alpha$ and $GS(\alpha)$ contains two distinct ergodic ground states in the direction of $\alpha$.
\end{thmB}
The construction of the potential $\Phi$ in Theorem B uses the result of Chazottes and Hochman \cite{ChH} about the existence of a one-dimensional potential with two distinct ergodic ground states. Our contribution to the proof of Theorem B is to extend this potential to a $m$-dimensional potential with the desired properties. Notice  that our construction works for rather general classes of  systems and also does not use the  specific form of the one-dimensional potential in \cite{ChH}. Therefore, our approach could be  applied to extend other one-dimensional phenomena to higher dimensions.

Finally, we consider non-exposed boundary points. In the case when $X$ is a one-sided shift map  we construct
a potential $\Phi:X\to \bR^2$ for which we are able to completely control the shape of the rotation set $\R(\Phi)$ (see Example 1). In particular, we can assure (see Proposition \ref{prop1}) that $\partial \R(\Phi)$ is an infinite polygon for which we are able to compute the coordinates of the vertices. Furthermore, we obtain the following (see Proposition \ref{symmetric}).

\begin{thmC}
Let $f:X\to X$ be a one-sided full shift over an alphabet with 4 symbols. Then there exists $\Phi\in C(X,\bR^2)$ such that $\R(\Phi)$ has a line segment face $F_\alpha(\Phi)$ with the following properties:
\begin{itemize}
\item $\partial \R(\Phi)$ is an infinite polygon;
\item  $GS(\alpha)=\{\mu_\alpha\}$, $\mu_\alpha$ is non-ergodic and $\rv(\mu_\alpha)\in {\rm int} F_\alpha(\Phi)$;
\item $\rv(GS(\alpha'))\cap F_\alpha(\Phi)=\emptyset$ for all direction vectors $\alpha'\not=\alpha$.
\end{itemize}
\end{thmC}
By slightly modifying the construction in Theorem C we can collapse the endpoints of the face $F_\alpha(\Phi)$ into one single point $w$. This point $w$ then becomes a smooth exposed boundary point that does not contain a non-ergodic ground state in its rotation class (see Proposition \ref{non-ergodic}).

This paper is organized as follows. In Section 2 we review some
basic concepts and results about ground states, zero-temperature measures and  the thermodynamic formalism.
Section 3 is devoted to the proof of    Theorem A.
 In Section 4 we discuss ground states that belong to the rotation classes of exposed boundary points; in particular we present the proof of Theorem B.  Finally,
in Section 5 we construct an example of a rotation set whose boundary is an infinite polygon and show that this example displays several possibilities for ground states on line segments.

\section{Preliminaries}\label{sec:2}
In this section we discuss relevant background material which will be used later on. We will continue to use the notations from Section 1.
Let $f:X\to X$ be a continuous map on a compact metric space $(X,d)$. Let $\cM$ denote the space of all $f$-invariant Borel probability measures on $X$ endowed with the weak$^\ast$ topology. This makes $\cM$ a compact convex space. Moreover, let $\cM_E\subset \cM$ be the subset of ergodic measures.
In this paper we use as a standing assumption
that $f$ has finite topological entropy and
that the entropy map
$\mu\mapsto h_{\mu}(f)$
is upper semi-continuous on $\cM$. Here $h_{\mu}(f)$ denotes the measure-theoretic entropy of $f$ with respect to $\mu$ (see \cite{Wal:81} for the definition and details).

\subsection{Thermodynamic formalism. }

Given a continuous one-dimensional potential $\varphi:X\to \bR $ we denote the topological pressure of $\varphi$ (with respect to $f$) by $P_{\rm top}(\varphi)$ and the topological entropy of $f$ by $h_{\rm top}(f)$ (see~\cite{Wal:81} for the definition and further details). The topological pressure satisfies the well-known variational principle, namely,

\begin{equation}\label{eqvarpri}
P_{\rm top}(\varphi)=
\sup_{\mu\in \cM} \left(h_\mu(f)+\int_X \varphi\,d\mu\right).
\end{equation}
A measure
$\mu\in \cM$ that attains the supremum in \eqref{eqvarpri} is
called an equilibrium state (or equilibrium measure) of the potential $\varphi$. We denote by $ES(\varphi)$ the set of all equilibrium states of $\varphi$. Note that $ES(\varphi)$ is a compact and convex subset of $\cM$.
Moreover, our standing assumption that the entropy map
is upper semi-continuous implies that $ES(\varphi)\not=\emptyset$, in particular $ES(\varphi)$ contains at least one ergodic equilibrium state.

\subsection{Ground states and zero-temperature measures. }

Next we give the definition for ground states as well as for zero-temperature measures.

Let $\varphi\in C(X,\bR)$ be a continuous potential. We say $\mu\in \cM$ is a ground state of the
potential $\varphi$ if there exists an increasing sequence $(t_n)_n\subset \bR$ with $t_n\to \infty$, and a corresponding sequence of equilibrium states  $(\mu_{t_n})_n$  such that $\mu_{t_n}\in ES(t_n\varphi)$ and $\mu_{t_n}\to \mu$ as $n\to\infty$. Here we think of $t$ as the inverse temperature and $T=1/t$ as the temperature of the system. This means that the measure $\mu$ is an accumulation point of equilibrium states when the temperature approaches zero.
We denote by $GS(\varphi)$ the set of all ground states of $\varphi$.

In order to define zero-temperature measures we  require convergence of the measures $\mu_{t\varphi}$ rather than only convergence of a subsequence.
Namely, suppose there exists $t_0\geq 0$ such that  for all $t\geq t_0$ the potential $t\varphi$ has a unique equilibrium state $\mu_{t\varphi}$ (i.e. $ES(t\varphi)$ is a singleton). We say  $\mu\in \cM$ is a zero-temperature measure of the potential $\varphi$ if $\mu$ is the weak$^\ast$  limit of the measures $\mu_{t\varphi}$.
Uniqueness of the equilibrium states is known to hold for certain classes of systems and potentials including Axiom A systems,  subshifts of finite type and expansive homeomorphisms with specification and H\"older continuous potentials.
Recently, there has been significant progress in generalizing  uniqueness result for equilibrium states to wider classes of shift transformations, non-uniformly hyperbolic maps and flows (e.g. \cite{CT1}, \cite{CT2}, \cite{CT3} and \cite{CFT}).  We refer to the overview article \cite{CP} for further references and details.

Notice that the fact that each $t\varphi$ has a unique equilibrium state does  in general not guarantee the existence of a zero-temperature measure. This is shown by Chazottes and  Hochman in \cite{ChH}. They construct a subshift of finite type and a
Lipschitz continuous potential $\varphi$ such that $GS(\varphi)$ contains two distinct ergodic measures. On the other hand, results about the convergence of the sequence $\mu_{t\varphi}$ are known only for special classes of system (subshifts of finite type and expanding systems) and potentials (see \cite{Br}, \cite{C}, \cite{Le}). We consider in this paper general classes of systems and in particular only require upper-semi continuity of the entropy map with the main focus on ground states.

\subsection{Rotation sets and entropy}
We now introduce generalized rotation sets and entropy of rotation vectors. We refer to \cite{KW1,Je} and the references therein for details and further accounts.
For $\Phi=(\phi_1,\cdots,\phi_m)\in C(X,\bR^m)$
we define the {\it rotation set} of  $\Phi$ and $f$  by
\begin{equation} \label{defrotset}
 \R(\Phi,f)= \left\{\rv(\mu): \mu\in\cM\right\},
\end{equation}
where $\rv(\mu)=\left(\int \phi_1\ d\mu,\ldots,\int \phi_m\ d\mu\right)$
denotes the rotation vector of the measure $\mu$.  Since we will always work with a fixed dynamical system $(X,f)$ we frequently omit the dependence on $f$ in the notation of the rotation set and write $\R(\Phi)$ instead of $\R(\Phi,f)$.
For $w\in \R(\Phi)$ we call $\cM_\Phi(w)=\{\mu\in \cM: \rv(\mu)=w\}$ the rotation class of $w$. Similarly, for a subset $F\subset\R(\Phi)$ we call $\cM_\Phi(F)=\{\mu\in \cM: \rv(\mu)\in F\}$ the rotation class of $F$.

Next we define the entropy of rotation vectors.
Following \cite{Je, KW1} we define the entropy of $w\in \R(\Phi)$ by
\begin{equation}\label{defH}
\cH(w)\eqdef\sup\{h_\mu(f): \mu \in \cM\ {\rm and }\ \rv(\mu)=w\}.
\end{equation}
The number $\cH(w)$ is also called localized entropy of $w$ (see \cite{KW2}). It follows from the upper-semi continuity of $\mu\mapsto h_\mu(f)$ that the supremum in \eqref{defH} is attained by at least one invariant measure and we call such a measure $\mu$ a localized measure of maximal entropy at $w$. Further, the map $w\mapsto \cH(w)$ is continuous on $\R(\Phi)$, see \cite{KW1}.

\subsection{Notions from convex geometry. }

Next we recall  some notions from convex geometry (see e.g. \cite[Ch. 2]{Gr}). For $m$-dimensional vectors $u=(u_1,...,u_m)$ and $v=(v_1,...,v_m)$ we write $u\cdot v =u_1v_1+...+u_mv_m$ for the inner product of $u$ and $v$. We also write ${\rm pr}_i(v)$ for the $i$-${\rm th}$  coordinate of $v$. Let $B(v,r)$ denote the open  ball about $v\in \bR^m$ of radius $r$ with respect to the Euclidean metric. A subset $K\in\bR^m$ is \emph{convex} if for every $u, v\in K$ we have that $tu+(1-t)v\in K$ for all $t\in(0,1)$. A point $w\in K$ is called an \emph{extreme} point of $K$ if $w=tu+(1-t)v$ for some $t\in(0,1)$ and  some $u, v\in K$ implies $u=v=w$.

We will work with the standard topology on $\bR^m$.
For $K\subset\bR^m$ we denote by $int(K)$ the interior of $K$ and by $\partial K$ the boundary of $K$.  The \emph{relative interior}, denoted by $ri(K)$, is the interior of $K$ with respect to the topology of the smallest affine subspace of $\bR^m$ containing $K$.

For a non-zero vector $\alpha\in\bR^m$ and $a\in \bR$ the hyperplane $H=H_{\alpha,a}\eqdef\{u\in \bR^m: u\cdot \alpha= a\}$ is said to \emph{cut} $K$ if both open half spaces determined by $H$ contain points of $K$. Here $\alpha$ is a normal vector to $H$. We say that $H$ is a \emph{supporting hyperplane} for $K$ if its distance to $K$ is zero but it does not cut $K$.

A set $F\subset K$ is a \emph{face} of $K$ if there exist a supporting hyperplane $H$ such that $F=K\cap H.$ We say a normal vector $\alpha$ to $H$ is pointing away from $K$ if for $w\in H$ the point $\alpha+w$ belongs to the open half space of $\bR^m\setminus H$ that does not intersect  $K$. We note that if $K$ has a non-empty interior then there exists a unique unit normal vector to $H$ that is pointing away from $K$.
A point $w\in K$ is called \emph{exposed} if $\{w\}$ is a face of $K$. We say  $w\in\partial K$ is smooth if the supporting hyperplane of $K$ that contains $w$ is unique. A compact convex set is strictly convex if all its boundary points are exposed.
Every exposed point of $K$ is an extreme point, but not vice versa. For example, consider a set $K=\{(v_1,v_2)\in\bR^2: v_1^2+v_2^2\le1\}\cup\{(v_1,v_2)\in\bR^2: -1\le v_1\le 1, 0\le v_2\le 1\}$. Then the points $(-1,0)$ and $(1,0)$ are extreme but not exposed.

We will split the boundary points of a given rotation set in three groups: exposed points, extreme non-exposed points and non-extreme. In the following sections we  analyze  these groups of boundary points  and  derive results about the existence/non-existence of  ground states and zero-temperature measures in their rotation
classes.

\section{Ground states at the boundary of rotation sets. }

In this section we study the relation between ground states and entropy-maximizing measures at the boundary of the rotation set.
As before, we consider a continuous dynamical system $f:X\to X$ with finite topological entropy and an upper semi-continuous entropy map.
We start by introducing some notation.

Let $\Phi\in C(X,\bR^m)$ be fixed, and let $\R(\Phi)$ be the corresponding rotation set of $\Phi$. We call $S^{m-1}=\{\alpha\in \bR^m: \|\alpha\|=1\}$ the set of direction vectors. Given a direction vector $\alpha$ we denote by $H_\alpha(\Phi)$ the supporting hyperplane of $\R(\Phi)$ for which $\alpha$ is the normal vector that points away from $\R(\Phi)$.
Since $\R(\Phi)$ is a compact convex set, it follows from standard arguments in convex geometry that $H_\alpha(\Phi)$ is uniquely defined.
For $w\in \partial \R(\Phi)$ we write $\gamma(w)=\{\alpha: w\in H_\alpha(\Phi)\}$ and call $\gamma(w)$ the set of direction vectors associated with $w$.
We denote $F
_\alpha(\Phi)\eqdef\R(\Phi)\cap H_\alpha(\Phi)$ for the face of the rotation set associated with the direction vector $\alpha$.
We write $\alpha\cdot \Phi(x)=\alpha_1 \phi_1(x)+\cdots+
\alpha_m \phi_m(x)$ and observe that $\alpha\cdot \Phi$ is a one-dimensional continuous potential. If $\mu \in GS(\alpha)\eqdef GS(\alpha\cdot\Phi)$ we say that $\mu$ is a  ground state in the direction of $\alpha$. Evidently $GS(\alpha)\not=\emptyset$, and  $\mu\in GS(\alpha)$  is a zero-temperature measure of the potential $\alpha\cdot \Phi$ if and only if $GS(\alpha)$ is a singleton. In this case we write $ZTM(\alpha)$ for $GS(\alpha)$.
We call $GS(\Phi)\eqdef\bigcup_{\alpha\in S^{m-1}} GS(\alpha)$ the set of  ground states of $\Phi$ and denote by $ZTM(\Phi)\subset GS(\Phi)$ the subset of zero temperature measures of $\Phi$.

The following theorem states that all ground states are entropy-maximizing measures at the boundary of the rotation set of $\Phi$. More precisely,  every ground state in the direction of $\alpha$ has its rotation vector on the supporting hyperplane $H_\alpha(\Phi)$ and maximizes entropy among all measures in the rotation class of the face $\R(\Phi)\cap H_\alpha(\Phi)$. In particular, every face of $\R(\Phi)$ contains at least one rotation vector with a corresponding ground state and if a face contains multiple ground states, they all have the same entropy. We further show that the set of all ground states in the direction $\alpha$ is a compact and (under mild additional assumptions) connected set.
\begin{theorem}\label{thm1} Let $f$ be a continuous map on a compact metric space for which the entropy map is upper-semi continuous. Let $\Phi:X\to\bR^m$ be a continuous potential and let $\alpha$ be a direction vector. Then
\begin{enumerate}
\item[(a)]
If  $\mu$ is a ground state in the direction of $\alpha$ then $\rv(\mu)\in H_\alpha(\Phi)$ and $h_\mu(f)= \sup\left\{h_\nu(f): \rv(\nu)\in F_\alpha(\Phi)\right\}$.
\item[(b)]
The set $\{\rv(\mu): \mu \in GS(\alpha)\}$ is compact.  Further, if for large enough $t$ there exist unique equilibrium states $\mu_t$
of the potential $t\alpha\cdot \Phi$ and $t\mapsto \rv(\mu_t)$ is continuous then $\{\rv(\mu): \mu \in GS(\alpha)\}$ is connected.
\end{enumerate}
\end{theorem}
\begin{proof}
We first prove  (a).
 Using that $\alpha$ is the unique unit normal vector of $H_\alpha(\Phi)$ that points away from $\R(\Phi)$ we may conclude for all $w\in H_\alpha(\Phi)$ that
\begin{equation}\label{eqsp}
\R(\Phi)\setminus H_\alpha(\Phi)\subset\{v\in \bR^m: \alpha\cdot(w-v)> 0\}.
\end{equation}
Let $U$ be any open set that contains the face $F_\alpha(\Phi)=\R(\Phi)\cap H_\alpha(\Phi)$. Then $\R(\Phi)\setminus U$ is a compact set which does not intercept $H_\alpha(\Phi)$ and consequently $\varepsilon\eqdef\text{dist}(\R(\Phi)\setminus U,H_\alpha(\Phi))>0$. Let $\mu, \nu\in \cM$ be  measures with $\rv(\mu)\in H_\alpha(\Phi)$ and $\rv(\nu)\notin U$. Let $t>\frac{2}{\varepsilon}h_{\rm top}(f)$. Applying \eqref{eqvarpri} and \eqref{eqsp} we obtain
\begin{align*}
P_{\rm top}(t\alpha\cdot\Phi)&\ge h_\mu(f) +\int t\alpha\cdot\Phi\,d\mu\\
&\ge h_\nu(f)-h_{\rm top}(f)+ t\alpha\cdot\int \Phi\,d\mu\\
&= h_\nu(f)-h_{\rm top}(f)+ t\alpha\cdot\int \Phi\,d\nu+t\alpha\cdot( \rv(\mu)-\rv(\nu))\\
&\ge h_\nu(f)-h_{\rm top}(f)+ t\alpha\cdot\int \Phi\,d\nu+t\,\text{dist}(\rv(\nu),H_\alpha(\Phi))\\
&\ge h_\nu(f)-h_{\rm top}(f)+ t\alpha\cdot\int \Phi\,d\nu+t\varepsilon\\
&> h_\nu(f)+ t\alpha\cdot\int \Phi\,d\nu+h_{\rm top}(f).
\end{align*}
Hence
\begin{equation}
P_{\rm top}(t\alpha\cdot\Phi)-\left(h_\nu(f) +\int t\alpha\cdot\Phi\,d\nu\right)>h_{\rm top}(f),
\end{equation}
 and we conclude that $\nu$ is not an equilibrium state of $t\alpha\cdot\Phi$. It follows that  for $t>\frac{2}{\varepsilon}h_{\rm top}(f)$ the rotation vectors of all equilibrium states of the potentials $t\alpha\cdot\Phi$  must be contained in $U$, and thus $\rv(GS(\alpha))\subset U$. Since $U$ was an arbitrary open set containing $F_\alpha(\Phi)$, we conclude that $\rv(GS(\alpha))\subset F_\alpha(\Phi)$.

Next we show that $\mu\in GS(\alpha)$ maximizes entropy among the invariant measures with rotation vectors in $F_\alpha(\Phi)$. Since $\mu$ is a  ground state in the direction of $\alpha$    there exist an increasing sequence $(t_n)_n\subset \bR$ with $t_n \to\infty$ and a corresponding sequence of equilibrium states  $(\mu_{t_n})_n\subset   \bigcup_n ES(t_n\alpha\cdot \Phi)$ such that $\mu_{t_n}\to \mu$ as $n\to \infty$. Hence,
 $\rv(\mu_{t_n})\to \rv(\mu)$ as $n\to \infty$. For each $t_n$ we consider the hyperplane  $H(t_n)=H_{\alpha,\rv(\mu_{t_n})}\eqdef\{w\in \bR^m: \alpha\cdot w=\alpha\cdot \rv(\mu_{t_n})\}$. We note that $H(t_n)$ and $H_\alpha(\Phi)$ are parallel hyperplanes with distance
 at most $||\rv(\mu_{t_n})-\rv(\mu)||$.

We claim that $h_{\mu_{t_n}}(f)\geq \cH(w)$ for all $w\in H(t_n)\cap \R(\Phi)$. Let $w\in H(t_n)\cap \R(\Phi)$ and
let $\nu\in \cM_\Phi(w)$. Since $\alpha\cdot w = \alpha \cdot \rv(\mu_{t_n})$ we conclude that
\[
\int t_n \alpha\cdot \Phi d\nu=t_n \alpha \cdot \rv(\nu)=t_n\alpha\cdot w =t_n \alpha \cdot \rv(\mu_{t_n})=\int t_n \alpha\cdot \Phi d\mu_{t_n}.
\]
On the other hand, since $\mu_{t_n}$ is an equilibrium state of the potential $t_n\alpha\cdot \Phi$ it follows from \eqref{eqvarpri} that $h_{\mu_{t_n}}(f)\geq h_{\nu}(f)$. Since $\nu\in \cM_\Phi(w)$ was arbitrary we obtain that $h_{\mu_{t_n}}(f)\geq \cH(w)$ which proves the claim.

Suppose $w\in F_\alpha(\Phi)$. If $H(t_1)$ cuts $\R(\Phi)$, there is a point $w_0\in\R(\Phi)$ which belongs to the other half space determined by $H(t_1)$ than $w$. Otherwise, let $w_0$ be any point in $ri (\R(\Phi))$. The line segment $L=\{sw_0+(1-s)w: s\in[0,\,1]\}\subset\R(\Phi)$ crosses $H(t_1)$, and hence is not parallel to any $H(t_n)$. Denote by $w_n$ the intersection points of $L$ and $H(t_n)$. By construction $dist(H(t_n),H_\alpha(\Phi))\to 0$ and hence $w_n\to w$. It follows from the above argument that $H(w_n)\le h_{\mu_{t_n}}(f)$. Now the continuity of the map $w\mapsto \cH(w)$ and the fact that  $\rv(\mu_{t_n})\to \rv(\mu)$ as $n\to\infty$ imply that $\cH(w)\le h_{\mu}(f)$. We conclude that $\mu$ is entropy-maximizing among the invariant measures with rotation vectors in the face $F_\alpha(\Phi)$.

Finally, we prove (b). The assertion that $\{\rv(\mu): \mu \in GS(\alpha)\}$ is a closed (and hence compact) subset of $\bR^m$ follows directly from the definition of $GS(\alpha)$. To complete the proof we still have to show that $\{\rv(\mu): \mu \in GS(\alpha)\}$ is connected.
Suppose that for large enough $t$ there exist unique equilibrium states $\mu_t$
of the potentials $t\alpha\cdot \Phi$ and that $t\mapsto \rv(\mu_t)$ is continuous. The case when $\{\rv(\mu): \mu \in GS(\alpha)\}$ contains only one rotation vector is trivial. Therefore, we can assume that there exist at least two distinct rotation vectors associated with the ground states in the direction of $\alpha$.
After a linear change of coordinates we may assume that $H_\alpha(\Phi)=\bR^{m-1}\times \{0\}$. Suppose $\{\rv(\mu): \mu \in GS(\alpha)\}$ is disconnected. Then there exist disjoint open sets $U_{m-1}, V_{m-1}\subset \bR^{m-1}$ such $(U_{m-1}\times \{0\}, V_{m-1}\times \{0\})$ forms a disconnection of $\{\rv(\mu): \mu \in GS(\alpha)\}$. In particular, there exist measures $\nu_1,\nu_2\in GS(\alpha)$ with $\rv(\nu_1)\in U_{m-1}\times \{0\}$ and  $\rv(\nu_2)\in V_{m-1}\times \{0\}$.

Given $\epsilon>0$ we define sets $U(\epsilon)=U_{m-1}\times (-\epsilon,\epsilon)$ and  $V(\epsilon)=V_{m-1}\times (-\epsilon,\epsilon)\subset \bR^m$. Clearly $U(\epsilon)$ and $V(\epsilon)$ are disjoint open sets in $\bR^m$ with $\rv(GS(\alpha))\subset U(\epsilon)\cup V(\epsilon)$.
It now follows from  (a) that there is $t_0>0$ such that $\rv(\mu_t)\in \bR^{m-1}\times (-\epsilon,\epsilon)$ for all $t\geq t_0$.
 Moreover, since $\nu_1, \nu_2$ are ground states in the direction of $\alpha$ there exists $t_0\leq t_1=t_1(\epsilon)< t_2=t_2(\epsilon)$ such that $\rv(\mu_{t_1})\in U(\epsilon)$ and $\rv(\mu_{t_2})\in V(\epsilon)$.

 Note that $ t\mapsto \rv(\mu_t)$ is  continuous  on $[t_1,t_2]$ with end points $\rv(\mu_{t_1})$ and $\rv(\mu_{t_2})$. We conclude that there exists $t_1<t_\epsilon<t_2$ such that
$$({\rm pr}_1(\rv(\mu_{t_\epsilon})),\cdots, {\rm pr}_{m-1}(\rv(\mu_{t_\epsilon})))\not\in (U_{m-1}\cup V_{m-1})\text{ and }|{\rm pr}_m(\rv(\mu_{t_\epsilon}))|\leq \epsilon.$$  Letting $\epsilon$ go to zero and applying the compactness of $\cM$ we can construct  a sequence $(\epsilon_n)_n$ with $\epsilon_n\to 0$  such that $\lim_{n\to\infty}\mu_{\epsilon_n}=\mu\in\cM$. It follows from the construction that  ${\rm pr}_m(\rv(\mu))=0$ and hence $\rv(\mu)\in H_\alpha(\Phi)\setminus (U(\epsilon)\cup V(\epsilon))$. This implies that
$\mu\in GS(\alpha)$ and $\rv(\mu)\not\in (U_{m-1}\cup V_{m-1})\times \{0\}$ which is a contradiction.
\end{proof}

\begin{remarks}
(i) As noted in Section 2.3 the property that the potentials $t\alpha\cdot \Phi$ have  unique equilibrium states holds for various classes of systems and potentials. In particular, Theorem \ref{thm1} (b) holds for subshifts of finite type, hyperbolic systems and expansive homeomorphisms with specification and H\"older continuous potentials. We note that in these cases the map $t\mapsto \rv(\mu_t)$ is real-analytic (see \cite{KW1}).

(ii) We note that the case  $|\rv(GS(\alpha))|\geq 2$  actually occurs. Indeed, we construct in Theorem  \ref{thm5}  a rotation set
with the  property that $GS(\alpha)$ is a non-trivial line segment.

\end{remarks}

As a consequence of Theorem \ref{thm1} (a) we obtain following:

\begin{corollary}\label{cor11} Let $f$ be a continuous map on a compact metric space for which the entropy map is upper-semi continuous. Let $\Phi:X\to\bR^m$ be a continuous potential and let $\alpha$ be a direction vector. Then the entropy function $\cH$ is constant
on $GS(\alpha)$.
\end{corollary}

\section{Exposed points. }\label{exposed_pts}

It follows from Theorem \ref{thm1} that  at an exposed point the corresponding rotation class necessarily  contains  a ground state. Moreover, all ground states in the corresponding direction must have the same rotation vector and the same entropy. One might suspect that this implies the existence of a zero-temperature measure. We show in Theorem \ref{thm3}  that contrary to the intuition such a statement does in general not hold in any dimension $m$.

First we summarise results from the previous sections applied to exposed points.
\begin{theorem}\label{thm2} Let $f$ be a continuous map on a compact metric space for which the entropy map is upper-semi continuous, and let $\Phi:X\to\bR^m$ be a continuous potential. Suppose $w\in \partial \R(\Phi)$ is an exposed point. Then
for every direction vector $\alpha\in \gamma(w)$ and every $\mu\in GS(\alpha)$  we have $\rv(\mu)=w$ and  $h_\mu(f)= \sup\left\{h_\nu(f): \rv(\nu)=w\right\}$.
\end{theorem}

\begin{proof}
The result follows immediately from Theorem \ref{thm1}.
\end{proof}
\begin{remark}
We note that by definition $\gamma(w)\not=\emptyset$. However $\gamma(w)$ is not necessarily a singleton. For example,
if $m=2$ and $w$ is a boundary point of $\R(\Phi)$ at which $\partial \R(\Phi)$ is not differentiable,  then $\gamma(w)$ is a non-empty interval in $S^1$ whose endpoints are the unit normal vectors to the right and left-hand side tangent lines at $w$.
\end{remark}

An immediate consequence of Theorem \ref{thm2} is the following.

\begin{corollary} \label{cor1}
Let $f$ be a continuous map on a compact metric space for which the entropy map is upper-semi continuous, and let $\Phi:X\to\bR^m$ be a continuous potential. Suppose $w\in \partial \R(\Phi)$ is an exposed point with a unique localized measure $\mu$ of maximal entropy at $w$. Then for every $\alpha\in\gamma(w)$ we have $GS(\alpha)=\{\mu\}$. In particular, if $ES(t\alpha\cdot \Phi)$ is a singleton for large enough $t$, then the sequence of equilibrium measures $\mu_{t\alpha\cdot \Phi}$ converges to $\mu$ as $t\to \infty$, that is $\mu$ is a zero-temperature measure in the direction of $\alpha$.
\end{corollary}

Corollary \ref{cor1} states that all the sequences $t\alpha\cdot \Phi$ converge to the same measure $\mu$ independently of $\alpha$. An example of a system that displays this phenomenon is given in \cite[Example 2]{KW1}  where we construct a rotation set $\R(\Phi)\subset \bR^2$ (associated with a one-sided shift map and a Lipschitz continuous potential $\Phi$) such that $\partial \R(\Phi)$ is a polygon with vertices $w_1,\cdots,w_k$ and $\cH(w_i)=\log 2$ for $i=1,\cdots,k$. Moreover,  for all  vertices $w_i$ there exists a unique localized measure $\mu_i$ of maximal entropy at $w_i$. Note that $\gamma(w_i)$ is a non-empty closed interval in $S^1$ and for all $\alpha$ in the interior of $\gamma(w_i)$ the measures
$\mu_{t\alpha\cdot \Phi}$ converge to $\mu_i$. This example can be easily modified so that $\cH(w_1)=\log 3$ and $\cH(w_i)=\log 2$ for $i=2,\cdots,k$. Then in view of Theorem \ref{thm1} we obtain that the measures $\mu_{t\alpha\cdot \Phi}$ converge to $\mu_1$ for all $\alpha\in\gamma(w_1)$. In this case, all points in the relative interior of the faces of the $\R(\Phi)$ adjacent to $w_1$ do not correspond to ground states.

The previous example displays an uncountable set of direction vectors all of which correspond to the same zero-temperature measure. In the following we establish the existence of the opposite phenomenon. Namely, we show that   an exposed point $w$ of $\R(\Phi)$  does in general not need to be attained by the rotation vector of a zero-temperature measure. Moreover, this phenomenon can occur even in the case when there is only one direction vector in $\gamma(w)$, i.e. $w$ is a smooth exposed point.

We will need the following.
We denote by $\Fix(f)$ the set of all fixed points of $f$, by $\Per_n(f)$ the set of all $\xi\in \Fix(f^n)$ and by $\Per(f)=\bigcup_n \Per_n(f)$  the set of all periodic points of $f$ . Given $\xi\in \Per_n(f)$ we denote by $\mu_{\xi}$ the unique invariant probability measure  supported on the orbit of $\xi$, that is,
\begin{equation}\label{defpermes}
\mu_\xi=\frac{1}{n}\sum_{k=0}^{n-1} \delta_{f^k(\xi)},
\end{equation}
where $\delta$ denotes the Dirac measure supported on the point. Moreover, we define the rotation vector of $\xi$ by
\begin{equation}
\rv(\xi)\eqdef \rv(\mu_\xi) = \frac{1}{n} \sum_{k=0}^{n-1} \Phi(f^k(\xi)).
\end{equation}

In the next theorem we construct an extension of a one-dimensional potential to $\bR^m$  that preserves one-dimensional ground states.

\begin{theorem}\label{thm3} Let $f$ be a continuous map on a compact metric space for which the entropy map is upper-semi continuous and assume that the periodic point measures are dense in $\cM$.
For a closed $f$-invariant set  $Y\subset X$ we define the potential $\varphi_Y:X\to\bR$  by $\varphi_Y(\xi)=dist(\xi,Y)$. Suppose that for all $t\geq 0$   the potential $t\varphi_Y$ has a unique  equilibrium state $\mu_{t\varphi_Y}$ and that $GS(\varphi_Y)$ contains $k$ distinct ergodic measures. Then for any $m\in\bN$ there exists a continuous potential $\Phi:X\to\bR^m$ with ${\rm pr}_1(\Phi)=\varphi_Y$ such that  $\R(\Phi)$ has a smooth exposed point $w$  that
coincides with the rotation vectors of $k$ distinct ergodic ground states.

\end{theorem}
\begin{proof}
We start by defining $\Phi=(\phi_1,...,\phi_m):X\to \bR^m$ with $\phi_1=\varphi_Y$. For this we define $\phi_i\quad (i=2,...,m)$ in the following way. The one-dimensional rotation set of $\varphi_Y$ is an interval $[0,\,a]$ for some $a>0$. Let $(x_n)\subset [0,\,a]$ be an exponentially decreasing sequence and let $\varepsilon_n=\frac{1}{4}(x_{n}-x_{n+1})$. In this case the ratio $\frac{\varepsilon_n}{x_n}$ is a positive constant less than $\frac14$. Since the periodic orbits are dense in $\cM$, for each $n$ there is a periodic point $\xi_n\in X$ such that $\rv_{\phi_1}(\xi_n)\in (x_n-\varepsilon_n,x_n+\varepsilon_n)$. We denote the smallest period of $\xi_n$ by $p_n$ and let $\Or_n=\{\xi_n,f(\xi_n),f^2(\xi_n),...,f^{p_n-1}(\xi_n)\}$.

We denote by $Z_n=\{\xi\in\Or_n: d(\xi,Y)\le x_n+\frac{1}{n}\}$ and let $Z=\left(\cup_{n=1}^\infty Z_n\right)\cup Y$. First we show that $Z$ is a closed subset of $X$. Suppose $(\zeta_j)$ is a convergent sequence in $Z$ and its limit is not in $Z_n$ for all $n\in\bN$. Since all sets $Z_n$ are finite, for any $l\in\bN$  there is $M>0$ such that  $\zeta_j$ is not in $\cup_{i=1}^l Z_i$ for all $j>M$. Then $dist(\zeta_j, Y)<x_l+\frac{1}{l}$ and as a consequence  $\lim_{j\to\infty}\zeta_j$ is in $Y$. This proves that $Z$ is closed.

We pick any sequence $(y_n)\subset \bR$  which satisfies
\begin{itemize}
  \item $0<y_n\le 1$
  \item $(y_n)$ monotonically decreases to zero
  \item $\lim_{n\to\infty} nx_n/y_n=0$
\end{itemize}
Note that we may take $y_n=n^{-\beta}$ for any $\beta>0$.  For $i=2,...,m$ we define
\begin{equation}\label{def_phi_i}
\phi_i(\xi)=\left\{
                         \begin{array}{ll}
                           (-1)^n y_n & \hbox{if $\xi\in Z_n$ and $(i-2)\equiv n \mod (m-1)$;} \\
                           0 & \hbox{if $\xi\in Z_n$ and $(i-2)\not\equiv n \mod (m-1)$;} \\
                           0 & \hbox{if $\xi\in Y$.}
                         \end{array}
                       \right.
\end{equation}
Since all limit points of $Z$ are in $Y$ and $y_n\to 0$, each $\phi_i:Z\to \bR$ is continuous. By the Tietze extension theorem we can extend $\phi_i$ to be a continuous function on $X$ in such a way that $\sup\{|\phi_i(\xi)|:\xi\in X\}\le 1$.

Next we estimate the values of the rotation vectors of $\xi_n$ using the information about their projections onto the first coordinate axis. Denote by $c_n$ the cardinality of $Z_n$ and let $k_n=\frac{c_n}{p_n}$. Note that $k_n\le 1$. We have
\begin{align*}
  {\rm pr}_1(\rv(\xi_n)) &= \frac{1}{p_n}\sum_{\xi\in\Or_n}\phi_1(\xi) \\
 &\ge \frac{1}{p_n}\sum_{\xi\notin Z_n}\phi_1(\xi)\\
&\ge \frac{p_n-c_n}{p_n}\left(x_n+\frac{1}{n}\right)\\
&=(1-k_n)\left(x_n+\frac{1}{n}\right)
\end{align*}
Since $\rv_{\phi_1}(\xi_n)\in (x_n-\varepsilon_n,x_n+\varepsilon_n )$, we obtain $(1-k_n)\left(x_n+\frac{1}{n}\right)\le x_n+\varepsilon_n$ and hence
\begin{equation}\label{k_n-bound}
  k_n\ge\frac{1-n\varepsilon_n}{1+nx_n}
\end{equation}
The exponential decay of the sequence $x_n$ implies that $k_n\to 1$ as $n\to \infty$.

From now on we assume that $n$ is even and obtain a lower bound for ${\rm pr}_i(\rv(\xi_n)),\, i=2,...,m$. In the case when $n$ is odd, the upper bound  for ${\rm pr}_i(\rv(\xi_n))$ could be obtained in a similar way. Using (\ref{k_n-bound}) for $i\in\{2,...,m\}$ and any even $n\in\bN$ satisfying $(i-2)\equiv n \mod (m-1)$ we have
\begin{align*}
 {\rm pr}_i(\rv(\xi_n)) &= \frac{1}{p_n}\sum_{\xi\in\Or_n}\phi_i(\xi) \\
 &= \frac{1}{p_n}\sum_{\xi\in Z_n}y_n+\frac{1}{p_n}\sum_{\xi\notin Z_n}\phi_i(\xi)\\
&\ge \frac{c_n}{p_n}y_n-\frac{p_n-c_n}{p_n}\\
&=k_ny_n-(1-k_n)\\
&\ge \frac{1-n\varepsilon_n}{nx_n+1}y_n-\frac{nx_n+n\varepsilon_n}{nx_n+1}\\
&= y_n\left(\frac{1-n\varepsilon_n-\frac{nx_n}{y_n}-\frac{n\varepsilon_n}{y_n}}{nx_n+1}\right)
\end{align*}
When $n$ increases, $n\varepsilon_n,\,\frac{nx_n}{y_n}$ and $\frac{n\varepsilon_n}{y_n}$ tend to zero, and hence the expression in parenthesis above approaches one. In particular, it follows that ${\rm pr}_i(\rv(\xi_n))\ge 0$ for large even $n$. In addition,
\begin{equation}\label{infty}
  \lim_{n\to\infty}\frac{{\rm pr}_i(\rv(\xi_{2n}))}{{\rm pr}_1(\rv(\xi_{2n}))}  \ge \lim_{n\to\infty}\frac{y_{2n}}{x_{2n}+\varepsilon_{2n}}=+\infty.
\end{equation}
Similarly, when $n$ is large and odd for $i\ge 2$ we have ${\rm pr}_i(\rv(\xi_n))<0$ and
\begin{equation}\label{-infty}
  \lim_{n\to\infty}\frac{{\rm pr}_i(\rv(\xi_{2n+1}))}{{\rm pr}_1(\rv(\xi_{2n+1}))}=-\infty.
\end{equation}

Next, we show that $\R(\Phi)$ has a unique supporting hyperplane at the origin. We denote the direction vector $(-1,0,...,0)$ by $\alpha$ and the origin by $w_0$. Note that $w_0\in\R(\Phi)$. Indeed, since $\Phi|_Y\equiv 0$, for any invariant measure $\mu$ supported on $Y$ we have $\rv(\mu)=\int\Phi\,d\mu=0$. Moreover, for any $w\in\R(\Phi)$ we take $\mu\in\cM_\Phi(w)$ and get
$$\alpha\cdot(w-w_0)=-{\rm pr}_1(w)=-\int\phi_1\,d\mu=-\int_X dist(\xi,Y)\,d\mu(\xi)\le 0.$$
Therefore, the hyperplane through $w_0$ orthogonal to $\alpha$  is a supporting hyperplane for $\R(\Phi)$ at $w_0$. Suppose there is another supporting hyperplane at $w_0$. Denote its normal vector by $\beta=(\beta_1,...,\beta_m)$.
Let $i\in \{2,3,...,m\}$ be such that $\beta_i\ne 0$; we may assume $\beta_i>0$. It follows from (\ref{infty}) and (\ref{-infty}) that there is an even natural $n_1$ and an odd natural $n_2$ such that $$(i-2)\equiv n_1 \mod (m-1)\equiv n_2 \mod (m-1)$$ and
\begin{itemize}
  \item $\beta_1{\rm pr}_1(\rv(\xi_{n_1}))+\beta_i{\rm pr}_i(\rv(\xi_{n_1}))>0$,
  \item $\beta_1{\rm pr}_1(\rv(\xi_{n_2}))+\beta_i{\rm pr}_i(\rv(\xi_{n_2}))<0$.
\end{itemize}
By (\ref{def_phi_i}) for any $j\in\{2,...,m\}\setminus\{i\}$ we have ${\rm pr}_j(\rv(\xi_{n_1}))={\rm pr}_j(\rv(\xi_{n_1}))=0$. Therefore, $\rv(\xi_{n_1})\cdot \beta>0 $ whereas $ \rv(\xi_{n_2})\cdot \beta<0 $ and hence the hyperplane with normal vector $\beta$ cuts $\R(\Phi)$. We conclude, that $\{u\in\bR^m:  \alpha\cdot u=0\}$ is the only supporting hyperplane of $\R(\Phi)$ at the origin.

Lastly, we show that the origin is an exposed point of $\R(\Phi)$. If a point $w\in \R(\Phi)$ is on the supporting hyperplane, then for some $\mu\in \cM_\Phi(w)$ $w=(\int\phi_1(\xi)\,d\mu,\int\phi_2(\xi)\,d\mu,...,\int\phi_m(\xi)\,d\mu )$  and the first coordinate of $w$ is zero. Therefore, for $\mu$-almost any $\xi$ we have $dist(\xi,Y)=\phi_1(\xi)=0$ and $\xi\in Y$. Since $\phi_i|_Y\equiv 0$ for i=2,...,m, we obtain that $w$ must be the origin.

For $\alpha=(-1,0,...,0)$ the potential $\alpha\cdot \Phi=\varphi_Y$. Therefore, the corresponding sequence of equilibrium states has $k$ weak$^*$ accumulation points $\mu_\alpha^1,...,\mu_\alpha^k$. Theorem \ref{thm2} now asserts that $\rv(\mu_\alpha^1)=...=\rv(\mu_\alpha^k)=w_0$.
\end{proof}

In \cite{ChH} Chazottes and Hochman construct a subshift $Y$ of the full shift $X=\{0,1\}^\bN$ with the following property: For the Lipschitz continuous potential $\varphi(\xi)=-dist(\xi,Y)$  the set $GS(\varphi_Y)$ contains two distinct ergodic invariant measures. Combining this result with Theorem \ref{thm3} yields the following.
\begin{corollary}\label{cor22}
Let $f:X\to X$ be the one-sided shift map over the alphabet $\{0,1\}$. Then for every $m\in \bN$ there exists a continuous potential $\Phi:X\to\bR^m$ such that  $\R(\Phi)$ has the origin as an exposed point, $\gamma(0)$ contains only one direction vector $\alpha$ and $GS(\alpha)$ contains two distinct ergodic ground states in the direction of $\alpha$.
\end{corollary}

\begin{remark}
The measures constructed in \cite{ChH} have zero entropy. However, one can easily obtain measures with positive entropy by considering products with a full shift.
The example in \cite{ChH} can be generalized so that $GS(\varphi_Y)$ contains any finite number of distinct ergodic invariant measures. We refer to \cite[Section 4]{ChH} for more details.

\end{remark}

On the other hand, it is possible to have a zero temperature measure at an exposed point that is non-ergodic. The corresponding example is given in Proposition \ref{non-ergodic} in the next section.

\section{Non-exposed points.}

The next example addresses the case when the boundary of a rotation set is not strictly convex. As we mentioned in Section 4, it is not difficult to construct examples where a non-extreme boundary point does not correspond to a ground state. Here we show that it may also happen at an extreme (non-exposed) boundary point.

\begin{example}\label{ex1}
Let $f:X\to X$ be the one-sided full shift  with alphabet $\{0,1,2,3\}$. For a real number $a>0$ and $i\in\{1,2\}$ let $\ell_i:[0,a]\to \bR$ be continuous functions such that $\ell_1$ is non-negative, increasing and strictly concave and $\ell_2$ is non-positive, decreasing and strictly convex. For $i=1,2$ we pick an exponentially decreasing to zero sequence $(x_i(k))_{k\in \bN}\subset (0,a)$. We denote the points on the graphs of $\ell_i$ corresponding to $x_i(k)$ by $v_i(k)=(x_i(k),\ell_i(x_i(k)))$. Also let $w_i(\infty)=(0,\ell_i(0))$ and $w(0)=(a,0)$. We refer to Figure \ref{DefPhi}.

Next, we define several subsets of $X$.
Let $S_1=\{0,1\}, S_{2}=\{2,3\}$ and fix $ \lambda\in \bN, \lambda\geq 3$.
For $i=1,2$ and all $k\geq \lambda$ we define $Y_i(k)=\{\xi \in X: \xi_1,\ldots, \xi_{k}\in S_i\}$ and $Y_i(\infty)=\{\xi \in X: \xi_j\in S_i\,\,\text{for all}\,\, j\}$. Note that  $Y_i(\infty)$ is a full shift on two symbols. Moreover, let $Y_0(\lambda)= X\setminus (Y_1(\lambda)\cup  Y_{2}(\lambda))$.
\\[0.2cm]
\noindent
Finally, we define a potential $\Phi: X\to \bR^2$ by
\begin{equation}\label{defpotphi}
\Phi(\xi)=\begin{cases}
w(0)\qquad & {\rm if}\,\,   \xi\in Y_0(\lambda)\\
                         v_{i}(k-\lambda) & {\rm if}\,\,  \xi\in Y_{i}(k-1)\ {\rm and}\ \xi\not\in Y_i(k),\, k> \lambda, \,i\in \{1,2\}\\
                          w_i(\infty) & {\rm if}\,\,   \  \xi\in Y_{i}(\infty),\,i\in \{1,2\}
            \end{cases}
\end{equation}
\end{example}

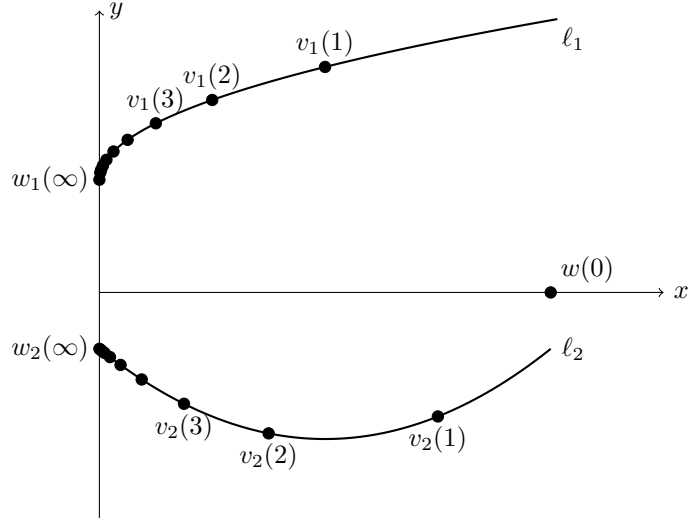
\begin{figure}[h]
\begin{center}
\begin{tikzpicture} [scale = 0.75]
\draw [color = black,->](0,0) --(10,0)node[right] {\small $x$};
\draw [color = black,->](0,-4) --(0,5)node[right] {\small $y$};
\draw [thick, domain = 2:4.85, samples=400] plot ( {(\x-2)^2}, {\x});
\draw [thick, domain = 0:8, samples=400] plot ( {\x}, {0.1*\x*(\x-8)-1});
{\draw [color = black,fill = black](0,2)  circle (.1) node[left] {\small $w_1(\infty)$};}
{\draw [color = black,fill = black](0,-1)  circle (.1) node[left] {\small $w_2(\infty)$};}
\draw [color = black,fill = black](8,0)  circle (.1) node[above right] {\small $w(0)$};
{\foreach \n in {1,...,3} \draw [color = black,fill = black]({8/(2^\n)},{2.828/(1.414)^\n+2})circle(.1) node[above] {\small $v_1(\n)$};
\foreach \n in {4,...,9} \draw [color = black,fill = black]({8/(2^\n)},{2.828/(1.414)^\n+2})  circle (.1);}
{\foreach \n in {1,...,3} \draw [color = black,fill = black]({12/(2^\n)},{0.1*12/(2^\n)*(12/(2^\n)-8)-1})circle(.1) node[below] {\small $v_2(\n)$};
\foreach \n in {4,...,9} \draw [color = black,fill = black]({12/(2^\n)},{0.1*12/(2^\n)*(12/(2^\n)-8)-1})circle(.1);}
\draw [color = black] (8.4,4.5)node{\small $\ell_1$};
\draw [color = black] (8.4,-1)node{\small $\ell_2$};
\end{tikzpicture}
\caption{ Construction of $\Phi$.}
\label{DefPhi}
\end{center}
\end{figure}

First we establish the continuity of $\Phi$.

\begin{lemma}\label{phi_continuity}
The potential $\Phi$ defined in Example \ref{ex1} is continuous.
\end{lemma}
\begin{proof}
Let $(\xi^n)$ be a convergent sequence in $X$ and $\xi=\lim \xi^n$. If $\Phi(\xi)=v_i(k)$ for some $i\in\{1,2\}$ and $k\in\bN$, then it follows immediately from (\ref{defpotphi}) that $\Phi(\xi^n)=v_i(k)$ for sufficiently large $n$.

Now suppose $\Phi(\xi)=w_i(\infty)$ for some $i\in\{1,2\}$. The continuity of $\ell_i$ implies that for a given $\varepsilon>0$ there is $k_0$ such that $v_i(k)\in B_\varepsilon(w_i(\infty))$ for all $k>k_0$. Since $\xi^n\to \xi$ we have that for $n$ large enough $\xi_1^n=\xi_1,\xi_2^n=\xi_2,\,...,\,\xi_{k_0}^n=\xi_{k_0}$. It follows that either $\xi^n\in Y_i(\infty)$ or $\xi^n\in Y_i(k) $ for some $k>k_0$. If the former is true, $\Phi(\xi^n)=w_i(\infty)$. Otherwise,  $\Phi(\xi^n)=v_i(k)$ for $k>k_0$ and hence $\Phi(\xi^n)\in B_\varepsilon(w_i(\infty))$.
\end{proof}

We denote by $w_i(j)$ the rotation vectors of the periodic orbits of length $j$ whose generators have the first $j-1$ coordinates in  $S_i$ and the $j^{\text{th}}$ coordinate in the complementary alphabet $S_{3-i}$. Precisely, for $j>\lambda$ and $i=1,2$ we have

\begin{equation}\label{wj} w_i(j)=\frac{\sum\limits_{k=1}^{j-\lambda}v_i(k)+\lambda w(0)}{j}.\end{equation}

\begin{lemma}\label{wj monotonic}
There is $j_0\ge \lambda$ such that the sequence of points $\{w_1(j)\}_{j>j_0}$ monotonically decreases to $w_1(\infty)$ and the sequence $\{w_2(j)\}_{j>j_0}$  monotonically increases to $w_2(\infty)$.
\end{lemma}
\begin{proof}
It follows from (\ref{wj}) that for any $j>\lambda$ and $i\in\{1,2\}$ we have
\begin{align}
w_i(j)-w_i(j+1)
&=\frac{1}{j(j+1)}\left[\lambda w(0)+\sum\limits_{k=1}^{j-\lambda}v_i(k)-jv_i(j+1-\lambda)\right].
\end{align}
The first coordinate of $w_i(j)-w_i(j+1)$ is always positive, since the $x_i(k)$ are decreasing and $w(0)=(a,0)$ with $a>x_i(j+1-\lambda)$. The second coordinate of ${w_i(j)-w_i(j+1)}$ simplifies to
\begin{equation}\label{y-coord wj}
\sum_{k=1}^{j-\lambda}\ell_i(x_i(k))-j\ell_i(x_i(j+1-\lambda)).
\end{equation}
For $i=1$ this expression is positive whenever $\ell_1(x_1(1))>(\lambda+1)\ell_1(x_1(j+1-\lambda))$. This can always be achieved starting from some $j_1$ since $\ell_1(x_1(k))$ is a decreasing sequence. Therefore, $w_1(j)$ are decreasing for $j>j_1$. Similarly, $w_2(j)$ are increasing for $j>j_2$ where $j_2$ is such that $\ell_2(x_2(1))<(\lambda+1)\ell_2(x_2(j_2+1-\lambda))$. Letting $j_0=\max\{j_1,j_2\}$ completes the proof.

\end{proof}

Next we show that the boundary of $\R(\Phi)$ is an infinite polygon (see Figure \ref{RotPhi}). Moreover, there is a neighborhood of the segment $[w_2(\infty),w_1(\infty)]$ where the vertices of $\R(\Phi)$ are exactly $w_i(j)$, $i=1,2$ and $j>j_0$ for some integer $j_0$ which depends on properties of the functions $\ell_i$. We prove this fact in the next proposition, where for simplicity we add an additional assumption on $\ell_i$ that guaranties that we may take $j_0=\lambda$.

\begin{proposition}\label{prop1} Let $\Phi$ be the potential defined in Example \ref{ex1} where, in addition, we have $\sum_{k=1}^\infty x_i(k)<a$ and $(-1)^i\ell_i(x_i(1))<(-1)^i(\lambda+1)\ell_i(x_i(2))$ for $i\in\{1,2\}$. Let $w_i(j)$ be as in (\ref{wj}) for $j>\lambda$ and set\\ $w_i(\lambda)=\frac{1}{3\lambda}[3(\lambda-1)w(0)+(3-i)v_i(1)+v_{3-i}(1)]$. Then \begin{equation}
\R(\Phi)=\overline{\rm Conv}\{w(0),w_i(j): j\ge\lambda,\,i=1,2\}.
\end{equation}
\end{proposition}

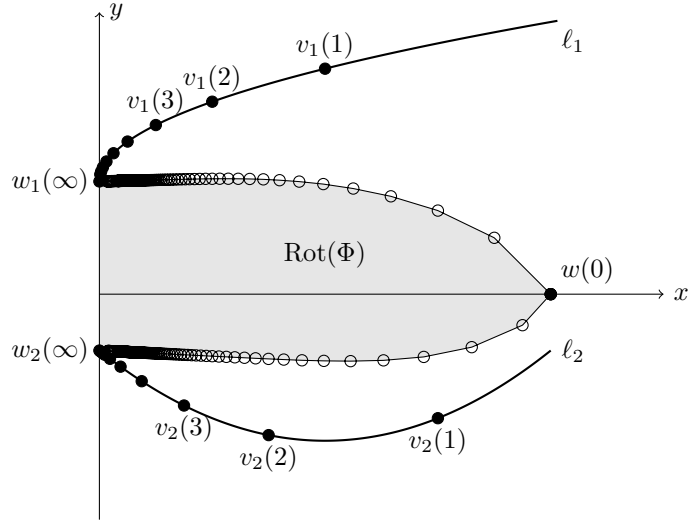
\begin{figure}[h]
\begin{center}
\begin{tikzpicture} [scale = 0.75]
\fill[fill=black!10] plot  (8,0)circle (.1)--
(7,1) --
(6,1.483) --
(5.167,1.736) --
(4.5,1.874) --
(3.969,1.953) --
(3.542,1.997) --
(3.194,2.022) --
(2.906,2.037) --
(2.665,2.044) --
(2.461,2.047) --
(2.285,2.048) --
(2.133,2.048) --
(2,2.047) --
(1.882,2.046) --
(1.778,2.044) --
(1.684,2.042) --
(1.6,2.04) --
(1.524,2.039) --
(1.455,2.037) --
(1.391,2.036) --
(1.333,2.034) --
(1.28,2.033) --
(1.231,2.032) --
(1.185,2.031) --
(1.143,2.03) --
(1.103,2.029) --
(1.067,2.028) --
(1.032,2.027) --
(1,2.026) --
(0.97,2.025) --
(0.941,2.024) --
(0.914,2.024) --
(0.889,2.023) --
(0.865,2.022) --
(0.842,2.022) --
(0.821,2.021) --
(0.8,2.021) --
(0.78,2.02) --
(0.762,2.02) --
(0.744,2.019) --
(0.727,2.019) --
(0.711,2.018) --
(0.696,2.018) --
(0.681,2.018) --
(0.667,2.017) --
(0.653,2.017) --
(0.64,2.017) --
(0.627,2.016) --
(0.615,2.016) --
(0.604,2.016) --
(0.593,2.015) --
(0.582,2.015) --
(0.571,2.015) --
(0.561,2.015) --
(0.552,2.014) --
(0.542,2.014) --
(0.533,2.014) --
(0.525,2.014) --
(0.516,2.013) --
(0.508,2.013) --
(0.5,2.013) --
(0.492,2.013) --
(0.485,2.013) --
(0.478,2.012) --
(0.471,2.012) --
(0.464,2.012) --
(0.457,2.012) --
(0.451,2.012) --
(0.444,2.012) --
(0.438,2.011) --
(0.432,2.011) --
(0.427,2.011) --
(0.421,2.011) --
(0.416,2.011) --
(0.41,2.011) --
(0.405,2.01) --
(0.4,2.01) --
(0.395,2.01) --
(0.39,2.01) --
(0.386,2.01) --
(0.381,2.01) --
(0.376,2.01) --
(0.372,2.01) --
(0.368,2.01) --
(0.364,2.009) --
(0.36,2.009) --
(0.356,2.009) --
(0.352,2.009) --
(0.348,2.009) --
(0.344,2.009) --
(0.34,2.009) --
(0.337,2.009) --
(0.333,2.009) --
(0.33,2.009) --
(0.327,2.008) --
(0.323,2.008) --
(0.32,2.008) --
(0.291,2.008) --
(0.267,2.007) --
(0.246,2.006) --
(0.229,2.006) --
(0.213,2.006) --
(0.2,2.005) --
(0.188,2.005) --
(0.178,2.005) --
(0.168,2.004) --
(0.16,2.004) --
(0.152,2.004) --
(0.145,2.004) --
(0.139,2.004) --
(0.133,2.003) --
(0,2)--
(0,0)--
(8,0) --
(7.5,-0.55) --
(6.6,-0.94) --
(5.75,-1.113) --
(5.036,-1.174) --
(4.453,-1.188) --
(3.979,-1.183) --
(3.591,-1.173) --
(3.268,-1.16) --
(2.998,-1.148) --
(2.768,-1.138) --
(2.571,-1.128) --
(2.4,-1.12) --
(2.25,-1.112) --
(2.118,-1.106) --
(2,-1.1) --
(1.895,-1.095) --
(1.8,-1.09) --
(1.714,-1.086) --
(1.636,-1.082) --
(1.565,-1.078) --
(1.5,-1.075) --
(1.44,-1.072) --
(1.385,-1.069) --
(1.333,-1.067) --
(1.286,-1.064) --
(1.241,-1.062) --
(1.2,-1.06) --
(1.161,-1.058) --
(1.125,-1.056) --
(1.091,-1.055) --
(1.059,-1.053) --
(1.029,-1.051) --
(1,-1.05) --
(0.973,-1.049) --
(0.947,-1.047) --
(0.923,-1.046) --
(0.9,-1.045) --
(0.878,-1.044) --
(0.857,-1.043) --
(0.837,-1.042) --
(0.818,-1.041) --
(0.8,-1.04) --
(0.783,-1.039) --
(0.766,-1.038) --
(0.75,-1.037) --
(0.735,-1.037) --
(0.72,-1.036) --
(0.706,-1.035) --
(0.692,-1.035) --
(0.679,-1.034) --
(0.667,-1.033) --
(0.655,-1.033) --
(0.643,-1.032) --
(0.632,-1.032) --
(0.621,-1.031) --
(0.61,-1.031) --
(0.6,-1.03) --
(0.59,-1.03) --
(0.581,-1.029) --
(0.571,-1.029) --
(0.563,-1.028) --
(0.554,-1.028) --
(0.545,-1.027) --
(0.537,-1.027) --
(0.529,-1.026) --
(0.522,-1.026) --
(0.514,-1.026) --
(0.507,-1.025) --
(0.5,-1.025) --
(0.493,-1.025) --
(0.486,-1.024) --
(0.48,-1.024) --
(0.474,-1.024) --
(0.468,-1.023) --
(0.462,-1.023) --
(0.456,-1.023) --
(0.45,-1.023) --
(0.444,-1.022) --
(0.439,-1.022) --
(0.434,-1.022) --
(0.429,-1.021) --
(0.424,-1.021) --
(0.419,-1.021) --
(0.414,-1.021) --
(0.409,-1.02) --
(0.404,-1.02) --
(0.4,-1.02) --
(0.396,-1.02) --
(0.391,-1.02) --
(0.387,-1.019) --
(0.383,-1.019) --
(0.379,-1.019) --
(0.375,-1.019) --
(0.371,-1.019) --
(0.367,-1.018) --
(0.364,-1.018) --
(0.36,-1.018) --
(0.327,-1.016) --
(0.3,-1.015) --
(0.277,-1.014) --
(0.257,-1.013) --
(0.24,-1.012) --
(0.225,-1.011) --
(0.212,-1.011) --
(0.2,-1.01) --
(0.189,-1.009) --
(0.18,-1.009) --
(0.171,-1.009) --
(0.164,-1.008) --
(0.157,-1.008) --
(0.15,-1.008)--
(0,-1)--
(0,0)
-- cycle;

\draw [color = black] plot
(8,0) circle (.1)--
(7,1)circle (.1)--
(6,1.483)circle (.1)--
(5.167,1.736)circle (.1)--
(4.5,1.874)circle (.1)--
(3.969,1.953)circle (.1)--
(3.542,1.997)circle (.1)--
(3.194,2.022)circle (.1)--
(2.906,2.037)circle (.1)--
(2.665,2.044)circle (.1)--
(2.461,2.047)circle (.1)--
(2.285,2.048)circle (.1)--
(2.133,2.048)circle (.1)--
(2,2.047)circle (.1)--
(1.882,2.046)circle (.1)--
(1.778,2.044)circle (.1)--
(1.684,2.042)circle (.1)--
(1.6,2.04)circle (.1)--
(1.524,2.039)circle (.1)--
(1.455,2.037)circle (.1)--
(1.391,2.036)circle (.1)--
(1.333,2.034)circle (.1)--
(1.28,2.033)circle (.1)--
(1.231,2.032)circle (.1)--
(1.185,2.031)circle (.1)--
(1.143,2.03)circle (.1)--
(1.103,2.029)circle (.1)--
(1.067,2.028)circle (.1)--
(1.032,2.027)circle (.1)--
(1,2.026)circle (.1)--
(0.97,2.025)circle (.1)--
(0.941,2.024)circle (.1)--
(0.914,2.024)circle (.1)--
(0.889,2.023)circle (.1)--
(0.865,2.022)circle (.1)--
(0.842,2.022)circle (.1)--
(0.821,2.021)circle (.1)--
(0.8,2.021)circle (.1)--
(0.78,2.02)circle (.1)--
(0.762,2.02)circle (.1)--
(0.744,2.019)circle (.1)--
(0.727,2.019)circle (.1)--
(0.711,2.018)circle (.1)--
(0.696,2.018)circle (.1)--
(0.681,2.018)circle (.1)--
(0.667,2.017)circle (.1)--
(0.653,2.017)circle (.1)--
(0.64,2.017)circle (.1)--
(0.627,2.016)circle (.1)--
(0.615,2.016)circle (.1)--
(0.604,2.016)circle (.1)--
(0.593,2.015)circle (.1)--
(0.582,2.015)circle (.1)--
(0.571,2.015)circle (.1)--
(0.561,2.015)circle (.1)--
(0.552,2.014)circle (.1)--
(0.542,2.014)circle (.1)--
(0.533,2.014)circle (.1)--
(0.525,2.014)circle (.1)--
(0.516,2.013)circle (.1)--
(0.508,2.013)circle (.1)--
(0.5,2.013)circle (.1)--
(0.492,2.013)circle (.1)--
(0.485,2.013)circle (.1)--
(0.478,2.012)circle (.1)--
(0.471,2.012)circle (.1)--
(0.464,2.012)circle (.1)--
(0.457,2.012)circle (.1)--
(0.451,2.012)circle (.1)--
(0.444,2.012)circle (.1)--
(0.438,2.011)circle (.1)--
(0.432,2.011)circle (.1)--
(0.427,2.011)circle (.1)--
(0.421,2.011)circle (.1)--
(0.416,2.011)circle (.1)--
(0.41,2.011)circle (.1)--
(0.405,2.01)circle (.1)--
(0.4,2.01)circle (.1)--
(0.395,2.01)circle (.1)--
(0.39,2.01)circle (.1)--
(0.386,2.01)circle (.1)--
(0.381,2.01)circle (.1)--
(0.376,2.01)circle (.1)--
(0.372,2.01)circle (.1)--
(0.368,2.01)circle (.1)--
(0.364,2.009)circle (.1)--
(0.36,2.009)circle (.1)--
(0.356,2.009)circle (.1)--
(0.352,2.009)circle (.1)--
(0.348,2.009)circle (.1)--
(0.344,2.009)circle (.1)--
(0.34,2.009)circle (.1)--
(0.337,2.009)circle (.1)--
(0.333,2.009)circle (.1)--
(0.33,2.009)circle (.1)--
(0.327,2.008)circle (.1)--
(0.323,2.008)circle (.1)--
(0.32,2.008)circle (.1)--
(0.291,2.008)circle (.1)--
(0.267,2.007)circle (.1)--
(0.246,2.006)circle (.1)--
(0.229,2.006)circle (.1)--
(0.213,2.006)circle (.1)--
(0.2,2.005)circle (.1)--
(0.188,2.005)circle (.1)--
(0.178,2.005)circle (.1)--
(0.168,2.004)circle (.1)--
(0.16,2.004)circle (.1)--
(0.152,2.004)circle (.1)--
(0.145,2.004)circle (.1)--
(0.139,2.004)circle (.1)--
(0.133,2.003)circle (.1)

(8,0)circle (.1)--
(7.5,-0.55)circle (.1)--
(6.6,-0.94)circle (.1)--
(5.75,-1.113)circle (.1)--
(5.036,-1.174)circle (.1)--
(4.453,-1.188)circle (.1)--
(3.979,-1.183)circle (.1)--
(3.591,-1.173)circle (.1)--
(3.268,-1.16)circle (.1)--
(2.998,-1.148)circle (.1)--
(2.768,-1.138)circle (.1)--
(2.571,-1.128)circle (.1)--
(2.4,-1.12)circle (.1)--
(2.25,-1.112)circle (.1)--
(2.118,-1.106)circle (.1)--
(2,-1.1)circle (.1)--
(1.895,-1.095)circle (.1)--
(1.8,-1.09)circle (.1)--
(1.714,-1.086)circle (.1)--
(1.636,-1.082)circle (.1)--
(1.565,-1.078)circle (.1)--
(1.5,-1.075)circle (.1)--
(1.44,-1.072)circle (.1)--
(1.385,-1.069)circle (.1)--
(1.333,-1.067)circle (.1)--
(1.286,-1.064)circle (.1)--
(1.241,-1.062)circle (.1)--
(1.2,-1.06)circle (.1)--
(1.161,-1.058)circle (.1)--
(1.125,-1.056)circle (.1)--
(1.091,-1.055)circle (.1)--
(1.059,-1.053)circle (.1)--
(1.029,-1.051)circle (.1)--
(1,-1.05)circle (.1)--
(0.973,-1.049)circle (.1)--
(0.947,-1.047)circle (.1)--
(0.923,-1.046)circle (.1)--
(0.9,-1.045)circle (.1)--
(0.878,-1.044)circle (.1)--
(0.857,-1.043)circle (.1)--
(0.837,-1.042)circle (.1)--
(0.818,-1.041)circle (.1)--
(0.8,-1.04)circle (.1)--
(0.783,-1.039)circle (.1)--
(0.766,-1.038)circle (.1)--
(0.75,-1.037)circle (.1)--
(0.735,-1.037)circle (.1)--
(0.72,-1.036)circle (.1)--
(0.706,-1.035)circle (.1)--
(0.692,-1.035)circle (.1)--
(0.679,-1.034)circle (.1)--
(0.667,-1.033)circle (.1)--
(0.655,-1.033)circle (.1)--
(0.643,-1.032)circle (.1)--
(0.632,-1.032)circle (.1)--
(0.621,-1.031)circle (.1)--
(0.61,-1.031)circle (.1)--
(0.6,-1.03)circle (.1)--
(0.59,-1.03)circle (.1)--
(0.581,-1.029)circle (.1)--
(0.571,-1.029)circle (.1)--
(0.563,-1.028)circle (.1)--
(0.554,-1.028)circle (.1)--
(0.545,-1.027)circle (.1)--
(0.537,-1.027)circle (.1)--
(0.529,-1.026)circle (.1)--
(0.522,-1.026)circle (.1)--
(0.514,-1.026)circle (.1)--
(0.507,-1.025)circle (.1)--
(0.5,-1.025)circle (.1)--
(0.493,-1.025)circle (.1)--
(0.486,-1.024)circle (.1)--
(0.48,-1.024)circle (.1)--
(0.474,-1.024)circle (.1)--
(0.468,-1.023)circle (.1)--
(0.462,-1.023)circle (.1)--
(0.456,-1.023)circle (.1)--
(0.45,-1.023)circle (.1)--
(0.444,-1.022)circle (.1)--
(0.439,-1.022)circle (.1)--
(0.434,-1.022)circle (.1)--
(0.429,-1.021)circle (.1)--
(0.424,-1.021)circle (.1)--
(0.419,-1.021)circle (.1)--
(0.414,-1.021)circle (.1)--
(0.409,-1.02)circle (.1)--
(0.404,-1.02)circle (.1)--
(0.4,-1.02)circle (.1)--
(0.396,-1.02)circle (.1)--
(0.391,-1.02)circle (.1)--
(0.387,-1.019)circle (.1)--
(0.383,-1.019)circle (.1)--
(0.379,-1.019)circle (.1)--
(0.375,-1.019)circle (.1)--
(0.371,-1.019)circle (.1)--
(0.367,-1.018)circle (.1)--
(0.364,-1.018)circle (.1)--
(0.36,-1.018)circle (.1)--
(0.327,-1.016)circle (.1)--
(0.3,-1.015)circle (.1)--
(0.277,-1.014)circle (.1)--
(0.257,-1.013)circle (.1)--
(0.24,-1.012)circle (.1)--
(0.225,-1.011)circle (.1)--
(0.212,-1.011)circle (.1)--
(0.2,-1.01)circle (.1)--
(0.189,-1.009)circle (.1)--
(0.18,-1.009)circle (.1)--
(0.171,-1.009)circle (.1)--
(0.164,-1.008)circle (.1)--
(0.157,-1.008)circle (.1)--
(0.15,-1.008)circle (.1);

\draw [color = black,->](0,0) --(10,0)node[right] {\small $x$};
\draw [color = black,->](0,-4) --(0,5)node[right] {\small $y$};
\draw [thick, domain = 2:4.85, samples=400] plot ( {(\x-2)^2}, {\x});
\draw [thick, domain = 0:8, samples=400] plot ( {\x}, {0.1*\x*(\x-8)-1});
{\draw [color = black,fill = black](0,2)  circle (.1) node[left] {\small $w_1(\infty)$};}
{\draw [color = black,fill = black](0,-1)  circle (.1) node[left] {\small $w_2(\infty)$};}
\draw [color = black,fill = black](8,0)  circle (.1) node[above right] {\small $w(0)$};
{\foreach \n in {1,...,3} \draw [color = black,fill = black]({8/(2^\n)},{2.828/(1.414)^\n+2})circle(.1) node[above] {\small $v_1(\n)$};
\foreach \n in {4,...,9} \draw [color = black,fill = black]({8/(2^\n)},{2.828/(1.414)^\n+2})  circle (.1);}
{\foreach \n in {1,...,3} \draw [color = black,fill = black]({12/(2^\n)},{0.1*12/(2^\n)*(12/(2^\n)-8)-1})circle(.1) node[below] {\small $v_2(\n)$};
\foreach \n in {4,...,9} \draw [color = black,fill = black]({12/(2^\n)},{0.1*12/(2^\n)*(12/(2^\n)-8)-1})circle(.1);}
\draw [color = black] (8.4,4.5)node{\small $\ell_1$};
\draw [color = black] (8.4,-1)node{\small $\ell_2$};
\draw [color = black] (4,0.7)node{\small $\R(\Phi)$};
\end{tikzpicture}
\caption{ $\R(\Phi)$ is an infinite polygon.}
\label{RotPhi}
\end{center}
\end{figure}

\begin{proof}
Note that Lemma \ref{wj monotonic} and the fact that $(-1)^i\ell_i(x_i(1))<(-1)^i(\lambda+1)\ell_i(x_i(2))$ assure that for each $i=1,2$ the sequence of points $\{w_i(j)\}_{j>\lambda}$ monotonically converges to $w_i(\infty)$.

The result of Sigmund that the periodic point measures are dense in $\cM$ reduces our considerations to rotation vectors of measures supported on periodic orbits.

Suppose $\xi\in X$ is a periodic point of period $n$. We may assume that $\xi=(\xi_1,...,\xi_n,...)$ and $(\xi_1,...,\xi_n)$ is maximally partitioned into $k$ blocks of sizes $n_1,...,n_k$ such that $n_1+...+n_k=n$, and each block exclusively contains elements of either $S_1$ or $S_2$. It follows from the construction of $\Phi$ that $n\cdot\rv(\xi)$ is the sum of blocks of vectors of the form
\begin{equation}\label{decomposition}
(\lambda-1)w(0)+\sum_{i=1}^{n_j-(\lambda-1)}v_s(i),
\end{equation}
for $n_j\ge\lambda$. Here  $s=1$ if the elements of $j^\text{th}$ block are from $S_1$ and $s=2$ if the elements of $j^\text{th}$ block are from $S_2$.
In case $n_j\le \lambda-1$ the block's contribution is $n_jw(0)$.

First we show that $\rv(\xi)\in \overline{\rm Conv}\{w(0),w_s(j): j\ge\lambda,\,s=1,2\}$ for $k=2$ and $n_1,n_2\ge\lambda$. In this case we have
$$\rv(\xi)=\frac{1}{n}\left(
  \begin{array}{c}
    2(\lambda-1)a+\sum\limits_{i=1}^{n_1-\lambda+1}x_1(i)+\sum\limits_{i=1}^{n_2-\lambda+1}x_2(i) \\
    \sum\limits_{i=1}^{n_1-\lambda+1}\ell_1(x_1(i))+\sum\limits_{i=1}^{n_2-\lambda+1}\ell_2(x_2(i)) \\
  \end{array}
\right)
$$
The second coordinate of the expression above is zero if and only if $n_1=n_2=\lambda$ and in this case the first coordinate is less than $a$. Hence, By symmetry we restrict ourselves to the case when the second coordinate of $\rv(\xi)$ is positive, that is $n_1>n_2$. We compare $\rv(\xi)$ with points $w_1(n)$ and $\frac{n_1}{n}w_1(n_1)+\frac{n_2}{n}w_1(n_2)$. Note that $$w_1(n)=\frac{1}{n}\left(
  \begin{array}{c}
    \lambda a+\sum\limits_{i=1}^{n-\lambda}x_1(i) \\
    \sum\limits_{i=1}^{n-\lambda}\ell_1(x_1(i)) \\
  \end{array}
\right),$$
$$ \frac{n_1}{n}w_1(n_1)+\frac{n_2}{n}w_1(n_2)=\frac{1}{n}\left(
  \begin{array}{c}
    2\lambda a+\sum\limits_{i=1}^{n_1-\lambda}x_1(i)+\sum\limits_{i=1}^{n_2-\lambda}x_1(i) \\
    \sum\limits_{i=1}^{n_1-\lambda}\ell_1(x_1(i))+\sum\limits_{i=1}^{n_2-\lambda}\ell_1(x_1(i)) \\
  \end{array}
\right)
$$
as long as $n_2>\lambda$. When $n_2=\lambda$ we have
$$ \frac{n_1}{n}w_1(n_1)+\frac{n_2}{n}w_1(n_2)=\frac{1}{n}\left(
  \begin{array}{c}
    2\lambda a-a+\sum\limits_{i=1}^{n_1-\lambda}x_1(i)+2x_1(1)+x_2(1)\\
    \sum\limits_{i=1}^{n_1-\lambda}\ell_1(x_1(i))+\ell_2(x_2(1)) \\
  \end{array}
\right)
.$$ Using the facts that $a>\sum_{i=1}^\infty x_s(i)$ for $s=1,2$ and $\lambda\ge 3$, for the first coordinates we obtain
\begin{equation}{\rm pr}_1(w_1(n))\le{\rm pr}_1\left(\rv(\xi)\right)\le{\rm pr}_1\left(\frac{n_1}{n}w_1(n_1)+\frac{n_2}{n}w_1(n_2)\right)\end{equation}
Since $\ell_2$ has negative values and $\ell_1(x_1(n_1-\lambda+1))\le \ell_1(x_1(1))$, for the second coordinates we obtain
\begin{equation}{\rm pr}_2\left(\frac{n_1}{n}w_1(n_1)+\frac{n_2}{n}w_1(n_2)\right)\ge{\rm pr}_2(w_1(n))>{\rm pr}_2(\rv(\xi)),\end{equation} and thus $\rv(\xi)\in \overline{\rm Conv}\{w(0),w_i(j): j\ge\lambda,\,i=1,2\}$.

The case $k=3$ is similar. We have $n=n_1+n_2+n_3$ with $n_1,n_2,n_3\ge\lambda$. By symmetry, we may assume that $\rv(\xi)$ has a nonnegative second coordinate and that we can write
\begin{equation}\rv(\xi)=\frac1n \left[3(\lambda-1)w_0+\sum\limits_{i=1}^{n_1-\lambda+1}v_1(i)+\sum\limits_{i=1}^{n_2-\lambda+1}v_2(i)+\sum\limits_{i=1}^{n_3-\lambda+1}v_1(i)\right],
\end{equation}
where $n_1\ge n_3$. We compare $\rv(\xi)$ with points $\frac{n_1+n_2-1}{n}w_1(n_1+n_2-1)+\frac{n_3+1}{n}w_1(n_3+1)$ and $\frac{n_1}{n}w_1(n_1)+\frac{n_2}{n}w_1(n_2)+\frac{n_3}{n}w_1(n_3)$. We have
\begin{multline}\rv(\xi)-\left[\frac{n_1+n_2-1}{n}w_1(n_1+n_2-1)+\frac{n_3+1}{n}w_1(n_3+1)\right]\\
=\frac{1}{n}\left(
  \begin{array}{c}
    (\lambda-3)a+\sum\limits_{i=n_1+n_2-\lambda}^{n_1-\lambda+1}x_1(i)+\sum\limits_{i=1}^{n_2-\lambda+1}x_2(i)\\
    \sum\limits_{i=1}^{n_2-\lambda+1}\ell_2(x_2(i))-\sum\limits_{i=n_1-\lambda+2}^{n_1-\lambda+n_2-1}\ell_1(x_1(i))
  \end{array}
  \right).
\end{multline}
Since the first coordinate of the difference is positive and the second is negative, the point $\frac{n_1+n_2-1}{n}w_1(n_1+n_2-1)+\frac{n_3+1}{n}w_1(n_3+1)$ is to the left and above of $\rv(\xi)$. To compare $\rv(\xi)$ with the other point we first consider the case when all $n_j$ are strictly greater than $\lambda$. Clearly,
\begin{multline}
{\rm pr}_1(\rv(\xi))-{\rm pr}_1\left(\frac{n_1}{n}w_1(n_1)+\frac{n_2}{n}w_1(n_2)+\frac{n_3}{n}w_1(n_3)\right)\\
=\frac1n[x_1(n_1-\lambda+1)+x_1(n_2-\lambda+1)+x_3(n_1-\lambda+1)-3a]
\end{multline}
is negative and the point $\frac{n_1}{n}w_1(n_1)+\frac{n_2}{n}w_1(n_2)+\frac{n_3}{n}w_1(n_3)$ is to the right of $\rv(\xi)$. On the other hand, using the facts that the function $\ell_2$ has negative values and $n_j>\lambda$ we obtain
\begin{multline}n\cdot{\rm pr}_2(\rv(\xi))-n\cdot{\rm pr}_2\left(\frac{n_1}{n}w_1(n_1)+\frac{n_2}{n}w_1(n_2)+\frac{n_3}{n}w_1(n_3)\right)\\
=\ell_1(x_1(n_1-\lambda+1))+\ell_1(x_1(n_3-\lambda+1))
+\sum\limits_{i=1}^{n_2-\lambda+1}\ell_2(x_2(i))-\sum\limits_{i=1}^{n_2-\lambda}\ell_1(x_1(i))\\
<2\ell_1(x_1(2))-\ell_1(x_1(1))<0.
\end{multline}
The last expression is negative since $\ell_1(x_1(1))>(\lambda+1)\ell_1(x_1(2))$ by the assumption on the function $\ell_1$. Hence, the point $\frac{n_1}{n}w_1(n_1)+\frac{n_2}{n}w_1(n_2)+\frac{n_3}{n}w_1(n_3)$ is above $\rv(\xi)$.  It follows that $\rv(\xi)\in \overline{\rm Conv}\{w(0),w_i(j): j\ge\lambda,\,i=1,2\}$ as long as $n_j>\lambda$ for all $j$. The case when some of $n_j$ are equal to $\lambda$ requires separate consideration since the formula for $w_1(\lambda)$ is different. However, the estimates could be done in a similar way and we omit them here.
We point out that in the case $n_1=n_2=n_3=\lambda$ we have $\rv(\xi)=w_1(\lambda)$.

To conclude the proof we notice that the rotation vector of any periodic orbit can be written as a convex combination of vectors described in the previous two cases and $w_0$.
\end{proof}

We now show that in the symmetric case all ground states in the direction of $(-1,0)$ belong to the rotation class of the mid point between $w_1(\infty)$ and $w_2(\infty)$. Theorem \ref{thm1} guaranties that all other points in the interior of the face $[w_1(\infty),w_2(\infty)]$ do not correspond to ground states.

\begin{proposition}\label{symmetric} Let $\Phi=(\phi_1,\phi_2)$ be the potential defined in Example \ref{ex1} where, in addition, we have $x_1(k)=x_2(k)$ for all $k\in\bN$ and $\ell_2(x)=-\ell_1(x)$ for $x\in[0,a]$ with $\ell_1(x_1(1))>(\lambda+1)\ell_1(x_1(2))$. Suppose $\alpha=(-1,0)$ and $\mu_\alpha\in GS(\alpha)$. Then $\rv(\mu_\alpha)=(0,0)$.
\end{proposition}
\begin{proof}
Since the segment connecting $w_1(\infty)$ and $w_2(\infty)$ is the face corresponding to $H_\alpha(\Phi)$, Theorem \ref{thm1} implies that the first coordinate of $\rv(\mu_\alpha)$ is zero. To show that $\rv(\mu_\alpha)=(0,0)$ we prove that $\rv(\mu_{t\alpha\cdot\Phi})$ have zero second coordinate for all $t>0$, i.e. $\int\phi_2\,d\mu_{t\phi_2}=0$.

We define a map $T:X\to X$ by $T(\xi)=\bar{\xi}$ where
\begin{equation}
\bar{\xi}_j=\left\{
  \begin{array}{ll}
    \xi_j+2, & \hbox{if $\xi_j\in S_1$;}  \\
    \xi_j-2, & \hbox{if $\xi_j\in S_2$.}
  \end{array}
\right.
\end{equation}
Note that if $\xi_j$ is in $S_i$ then $\bar{\xi}_j$ is in the complementary alphabet $S_{3-i}$ for $i\in\{1,2\}$. It follows that
\begin{itemize}
  \item $\xi\in X_0\,\qquad \Longleftrightarrow \quad \bar{\xi}\in X_0$;
  \item $\xi\in X_i(k)\quad \Longleftrightarrow \quad \bar{\xi}\in X_{3-i}(k)\quad \text{for}\, i\in\{1,2\}$.
\end{itemize}
The symmetry in the definition of $\Phi$ implies that $\Phi(\bar{\xi})=(\phi_1(\xi),-\phi_2(\xi))$.
Clearly, $T$ is a homeomorphism ($T=T^{-1}$) and $T\circ f=f\circ T$. For any invariant measure $\mu$ the dynamical systems $f:(X,\mu)\to(X,\mu)$ and $f:(X,\mu\circ T)\to(X,\mu\circ T)$ are metrically isomorphic. Therefore, $h_\mu(f)=h_{\mu\circ T}(f)$.

For $\alpha=(-1,0)$ measure $\mu_{t\alpha\cdot\Phi}$ is the equilibrium state for the potential $t\alpha\cdot\Phi=-t\phi_1$. Since $h_{\mu_{t\alpha\cdot\Phi}}(f)=h_{\mu_{t\alpha\cdot\Phi}\circ T}(f)$ and $\int\phi_1(\xi)\,d\mu_{t\alpha\cdot\Phi}\circ T(\xi)=\int\phi_1(\bar{\xi})\,d\mu_{t\alpha\cdot\Phi}(\xi)=\int\phi_1(\xi)\,d\mu_{t\alpha\cdot\Phi}(\xi)$, the measure $\mu_{t\alpha\cdot\Phi}\circ T$ is also an equilibrium state for the potential $-t\phi_1$. The uniqueness of equilibrium states implies  $\mu_{t\alpha\cdot\Phi}\circ T=\mu_{t\alpha\cdot\Phi}$. However, $$\int\phi_2(\xi)\,d\mu_{t\alpha\cdot\Phi}\circ T(\xi)=\int\phi_2(\bar{\xi})\,d\mu_{t\alpha\cdot\Phi}(\xi)=-\int\phi_2(\xi)\,d\mu_{t\alpha\cdot\Phi}(\xi)$$ and thus we must have $\int\phi_2\,d\mu_{t\alpha\cdot\Phi}=0$.

It follows that $\rv(\mu_{t\alpha\cdot\Phi}),$ for $t>0$ are on the $x$-axis. The rotation vector of their accumulation point is the intersection of the $x$-axis and the boundary of the rotation set of $\Phi$, which is $(0,0)$. All other points on the boundary strictly between $w_2(\infty)$ and $w_1(\infty)$ are not rotation vectors of the ground states of $\Phi$.


\end{proof}

Now we specify the functions $\ell_1$ and $\ell_2$ so that $w_1(\infty)$ and $w_2(\infty)$ are extreme non-exposed points of the rotation set of $\Phi$ (see Figure \ref{RotPhi2}). Then we apply Proposition \ref{symmetric} and conclude that even if the point on the boundary is extreme (but non-exposed) it might not correspond to a rotation vector of any ground state of $\Phi$.

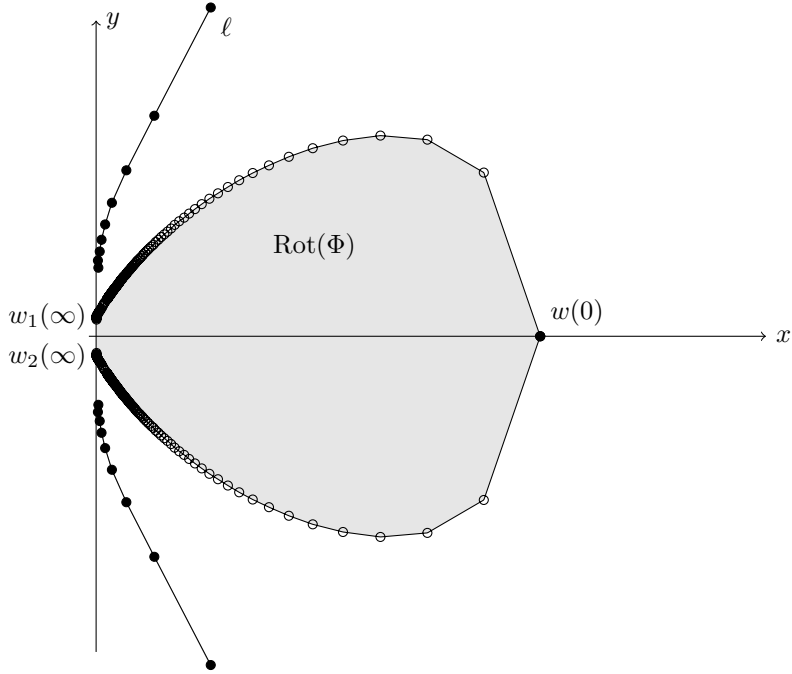
\begin{figure}[h]
\begin{center}
\begin{tikzpicture} [scale = 6]
\fill[fill=black!10] plot  (1,0)--
(0.875,-0.363)--
(0.75,-0.436)--
(0.646,-0.445)--
(0.563,-0.434)--
(0.496,-0.417)--
(0.443,-0.398)--
(0.399,-0.379)--
(0.363,-0.362)--
(0.333,-0.346)--
(0.308,-0.331)--
(0.286,-0.317)--
(0.267,-0.305)--
(0.25,-0.293)--
(0.235,-0.282)--
(0.222,-0.272)--
(0.211,-0.263)--
(0.2,-0.255)--
(0.19,-0.247)--
(0.182,-0.239)--
(0.174,-0.232)--
(0.167,-0.226)--
(0.16,-0.22)--
(0.154,-0.214)--
(0.148,-0.209)--
(0.143,-0.203)--
(0.138,-0.199)--
(0.133,-0.194)--
(0.129,-0.19)--
(0.125,-0.186)--
(0.121,-0.182)--
(0.118,-0.178)--
(0.114,-0.174)--
(0.111,-0.171)--
(0.108,-0.168)--
(0.105,-0.165)--
(0.103,-0.162)--
(0.1,-0.159)--
(0.098,-0.156)--
(0.095,-0.153)--
(0.093,-0.151)--
(0.091,-0.148)--
(0.089,-0.146)--
(0.087,-0.144)--
(0.085,-0.141)--
(0.083,-0.139)--
(0.082,-0.137)--
(0.08,-0.135)--
(0.078,-0.133)--
(0.077,-0.132)--
(0.075,-0.13)--
(0.074,-0.128)--
(0.073,-0.126)--
(0.071,-0.125)--
(0.07,-0.123)--
(0.069,-0.122)--
(0.068,-0.12)--
(0.067,-0.119)--
(0.066,-0.117)--
(0.065,-0.116)--
(0.063,-0.114)--
(0.063,-0.113)--
(0.062,-0.112)--
(0.061,-0.111)--
(0.06,-0.109)--
(0.059,-0.108)--
(0.058,-0.107)--
(0.057,-0.106)--
(0.056,-0.105)--
(0.056,-0.104)--
(0.055,-0.103)--
(0.054,-0.102)--
(0.053,-0.101)--
(0.053,-0.1)--
(0.052,-0.099)--
(0.051,-0.098)--
(0.051,-0.097)--
(0.05,-0.096)--
(0.049,-0.095)--
(0.049,-0.095)--
(0.048,-0.094)--
(0.048,-0.093)--
(0.047,-0.092)--
(0.047,-0.091)--
(0.046,-0.091)--
(0.045,-0.09)--
(0.045,-0.089)--
(0.044,-0.088)--
(0.044,-0.088)--
(0.043,-0.087)--
(0.043,-0.086)--
(0.043,-0.086)--
(0.042,-0.085)--
(0.042,-0.084)--
(0.041,-0.084)--
(0.041,-0.083)--
(0.04,-0.082)--
(0.04,-0.082)--
(0.036,-0.076)--
(0.033,-0.072)--
(0.031,-0.068)--
(0.029,-0.064)--
(0.027,-0.061)--
(0.025,-0.058)--
(0.024,-0.056)--
(0.022,-0.054)--
(0.021,-0.052)--
(0.02,-0.05)--
(0.019,-0.048)--
(0.018,-0.047)--
(0.017,-0.045)--
(0.017,-0.044)--
(0.0169,-0.043)--
(0.0169,-0.042)--
(0.0169,-0.041)--
(0.0169,-0.04)--
(0.0169,-0.037)--
(0.0169,0)--
(1,0)--
(0.875,0.363)--
(0.75,0.436)--
(0.646,0.445)--
(0.563,0.434)--
(0.496,0.417)--
(0.443,0.398)--
(0.399,0.379)--
(0.363,0.362)--
(0.333,0.346)--
(0.308,0.331)--
(0.286,0.317)--
(0.267,0.305)--
(0.25,0.293)--
(0.235,0.282)--
(0.222,0.272)--
(0.211,0.263)--
(0.2,0.255)--
(0.19,0.247)--
(0.182,0.239)--
(0.174,0.232)--
(0.167,0.226)--
(0.16,0.22)--
(0.154,0.214)--
(0.148,0.209)--
(0.143,0.203)--
(0.138,0.199)--
(0.133,0.194)--
(0.129,0.19)--
(0.125,0.186)--
(0.121,0.182)--
(0.118,0.178)--
(0.114,0.174)--
(0.111,0.171)--
(0.108,0.168)--
(0.105,0.165)--
(0.103,0.162)--
(0.1,0.159)--
(0.098,0.156)--
(0.095,0.153)--
(0.093,0.151)--
(0.091,0.148)--
(0.089,0.146)--
(0.087,0.144)--
(0.085,0.141)--
(0.083,0.139)--
(0.082,0.137)--
(0.08,0.135)--
(0.078,0.133)--
(0.077,0.132)--
(0.075,0.13)--
(0.074,0.128)--
(0.073,0.126)--
(0.071,0.125)--
(0.07,0.123)--
(0.069,0.122)--
(0.068,0.12)--
(0.067,0.119)--
(0.066,0.117)--
(0.065,0.116)--
(0.063,0.114)--
(0.063,0.113)--
(0.062,0.112)--
(0.061,0.111)--
(0.06,0.109)--
(0.059,0.108)--
(0.058,0.107)--
(0.057,0.106)--
(0.056,0.105)--
(0.056,0.104)--
(0.055,0.103)--
(0.054,0.102)--
(0.053,0.101)--
(0.053,0.1)--
(0.052,0.099)--
(0.051,0.098)--
(0.051,0.097)--
(0.05,0.096)--
(0.049,0.095)--
(0.049,0.095)--
(0.048,0.094)--
(0.048,0.093)--
(0.047,0.092)--
(0.047,0.091)--
(0.046,0.091)--
(0.045,0.09)--
(0.045,0.089)--
(0.044,0.088)--
(0.044,0.088)--
(0.043,0.087)--
(0.043,0.086)--
(0.043,0.086)--
(0.042,0.085)--
(0.042,0.084)--
(0.041,0.084)--
(0.041,0.083)--
(0.04,0.082)--
(0.04,0.082)--
(0.036,0.076)--
(0.033,0.072)--
(0.031,0.068)--
(0.029,0.064)--
(0.027,0.061)--
(0.025,0.058)--
(0.024,0.056)--
(0.022,0.054)--
(0.021,0.052)--
(0.02,0.05)--
(0.019,0.048)--
(0.018,0.047)--
(0.017,0.045)--
(0.017,0.044)--
(0.0169,0.043)--
(0.0169,0.042)--
(0.0169,0.041)--
(0.0169,0.04)--
(0.0169,0.037)--
(0.0169,0)
-- cycle;

\draw [color = black,fill = black] plot
(0.27,0.729)circle (.01)--
(0.145,0.489)circle (.01)--
(0.083,0.368)circle (.01)--
(0.051,0.296)circle (.01)--
(0.036,0.248)circle (.01)--
(0.028,0.214)circle (.01)--
(0.024,0.188)circle (.01)--
(0.02,0.168)circle (.01)--
(0.021,0.152)circle (.01)
;

\draw [color = black] plot
(1,0)circle (.01)--
(0.875,0.363)circle (.01)--
(0.75,0.436)circle (.01)--
(0.646,0.445)circle (.01)--
(0.563,0.434)circle (.01)--
(0.496,0.417)circle (.01)--
(0.443,0.398)circle (.01)--
(0.399,0.379)circle (.01)--
(0.363,0.362)circle (.01)--
(0.333,0.346)circle (.01)--
(0.308,0.331)circle (.01)--
(0.286,0.317)circle (.01)--
(0.267,0.305)circle (.01)--
(0.25,0.293)circle (.01)--
(0.235,0.282)circle (.01)--
(0.222,0.272)circle (.01)--
(0.211,0.263)circle (.01)--
(0.2,0.255)circle (.01)--
(0.19,0.247)circle (.01)--
(0.182,0.239)circle (.01)--
(0.174,0.232)circle (.01)--
(0.167,0.226)circle (.01)--
(0.16,0.22)circle (.01)--
(0.154,0.214)circle (.01)--
(0.148,0.209)circle (.01)--
(0.143,0.203)circle (.01)--
(0.138,0.199)circle (.01)--
(0.133,0.194)circle (.01)--
(0.129,0.19)circle (.01)--
(0.125,0.186)circle (.01)--
(0.121,0.182)circle (.01)--
(0.118,0.178)circle (.01)--
(0.114,0.174)circle (.01)--
(0.111,0.171)circle (.01)--
(0.108,0.168)circle (.01)--
(0.105,0.165)circle (.01)--
(0.103,0.162)circle (.01)--
(0.1,0.159)circle (.01)--
(0.098,0.156)circle (.01)--
(0.095,0.153)circle (.01)--
(0.093,0.151)circle (.01)--
(0.091,0.148)circle (.01)--
(0.089,0.146)circle (.01)--
(0.087,0.144)circle (.01)--
(0.085,0.141)circle (.01)--
(0.083,0.139)circle (.01)--
(0.082,0.137)circle (.01)--
(0.08,0.135)circle (.01)--
(0.078,0.133)circle (.01)--
(0.077,0.132)circle (.01)--
(0.075,0.13)circle (.01)--
(0.074,0.128)circle (.01)--
(0.073,0.126)circle (.01)--
(0.071,0.125)circle (.01)--
(0.07,0.123)circle (.01)--
(0.069,0.122)circle (.01)--
(0.068,0.12)circle (.01)--
(0.067,0.119)circle (.01)--
(0.066,0.117)circle (.01)--
(0.065,0.116)circle (.01)--
(0.063,0.114)circle (.01)--
(0.063,0.113)circle (.01)--
(0.062,0.112)circle (.01)--
(0.061,0.111)circle (.01)--
(0.06,0.109)circle (.01)--
(0.059,0.108)circle (.01)--
(0.058,0.107)circle (.01)--
(0.057,0.106)circle (.01)--
(0.056,0.105)circle (.01)--
(0.056,0.104)circle (.01)--
(0.055,0.103)circle (.01)--
(0.054,0.102)circle (.01)--
(0.053,0.101)circle (.01)--
(0.053,0.1)circle (.01)--
(0.052,0.099)circle (.01)--
(0.051,0.098)circle (.01)--
(0.051,0.097)circle (.01)--
(0.05,0.096)circle (.01)--
(0.049,0.095)circle (.01)--
(0.049,0.095)circle (.01)--
(0.048,0.094)circle (.01)--
(0.048,0.093)circle (.01)--
(0.047,0.092)circle (.01)--
(0.047,0.091)circle (.01)--
(0.046,0.091)circle (.01)--
(0.045,0.09)circle (.01)--
(0.045,0.089)circle (.01)--
(0.044,0.088)circle (.01)--
(0.044,0.088)circle (.01)--
(0.043,0.087)circle (.01)--
(0.043,0.086)circle (.01)--
(0.043,0.086)circle (.01)--
(0.042,0.085)circle (.01)--
(0.042,0.084)circle (.01)--
(0.041,0.084)circle (.01)--
(0.041,0.083)circle (.01)--
(0.04,0.082)circle (.01)--
(0.04,0.082)circle (.01)--
(0.036,0.076)circle (.01)--
(0.033,0.072)circle (.01)--
(0.031,0.068)circle (.01)--
(0.029,0.064)circle (.01)--
(0.027,0.061)circle (.01)--
(0.025,0.058)circle (.01)--
(0.024,0.056)circle (.01)--
(0.022,0.054)circle (.01)--
(0.021,0.052)circle (.01)--
(0.02,0.05)circle (.01)--
(0.019,0.048)circle (.01)--
(0.018,0.047)circle (.01)--
(0.017,0.045)circle (.01)--
(0.017,0.044)circle (.01)--
(0.0169,0.043)circle (.01)--
(0.0169,0.042)circle (.01)--
(0.0169,0.041)circle (.01)--
(0.0169,0.04)circle (.01)--
(0.0169,0.037)circle (.01)
;

\draw [color = black,fill = black] plot
(0.27,-0.729)circle (.01)--
(0.145,-0.489)circle (.01)--
(0.083,-0.368)circle (.01)--
(0.051,-0.296)circle (.01)--
(0.036,-0.248)circle (.01)--
(0.028,-0.214)circle (.01)--
(0.024,-0.188)circle (.01)--
(0.02,-0.168)circle (.01)--
(0.021,-0.152)circle (.01)
;

\draw [color = black] plot
(1,0)circle (.01)--
(0.875,-0.363)circle (.01)--
(0.75,-0.436)circle (.01)--
(0.646,-0.445)circle (.01)--
(0.563,-0.434)circle (.01)--
(0.496,-0.417)circle (.01)--
(0.443,-0.398)circle (.01)--
(0.399,-0.379)circle (.01)--
(0.363,-0.362)circle (.01)--
(0.333,-0.346)circle (.01)--
(0.308,-0.331)circle (.01)--
(0.286,-0.317)circle (.01)--
(0.267,-0.305)circle (.01)--
(0.25,-0.293)circle (.01)--
(0.235,-0.282)circle (.01)--
(0.222,-0.272)circle (.01)--
(0.211,-0.263)circle (.01)--
(0.2,-0.255)circle (.01)--
(0.19,-0.247)circle (.01)--
(0.182,-0.239)circle (.01)--
(0.174,-0.232)circle (.01)--
(0.167,-0.226)circle (.01)--
(0.16,-0.22)circle (.01)--
(0.154,-0.214)circle (.01)--
(0.148,-0.209)circle (.01)--
(0.143,-0.203)circle (.01)--
(0.138,-0.199)circle (.01)--
(0.133,-0.194)circle (.01)--
(0.129,-0.19)circle (.01)--
(0.125,-0.186)circle (.01)--
(0.121,-0.182)circle (.01)--
(0.118,-0.178)circle (.01)--
(0.114,-0.174)circle (.01)--
(0.111,-0.171)circle (.01)--
(0.108,-0.168)circle (.01)--
(0.105,-0.165)circle (.01)--
(0.103,-0.162)circle (.01)--
(0.1,-0.159)circle (.01)--
(0.098,-0.156)circle (.01)--
(0.095,-0.153)circle (.01)--
(0.093,-0.151)circle (.01)--
(0.091,-0.148)circle (.01)--
(0.089,-0.146)circle (.01)--
(0.087,-0.144)circle (.01)--
(0.085,-0.141)circle (.01)--
(0.083,-0.139)circle (.01)--
(0.082,-0.137)circle (.01)--
(0.08,-0.135)circle (.01)--
(0.078,-0.133)circle (.01)--
(0.077,-0.132)circle (.01)--
(0.075,-0.13)circle (.01)--
(0.074,-0.128)circle (.01)--
(0.073,-0.126)circle (.01)--
(0.071,-0.125)circle (.01)--
(0.07,-0.123)circle (.01)--
(0.069,-0.122)circle (.01)--
(0.068,-0.12)circle (.01)--
(0.067,-0.119)circle (.01)--
(0.066,-0.117)circle (.01)--
(0.065,-0.116)circle (.01)--
(0.063,-0.114)circle (.01)--
(0.063,-0.113)circle (.01)--
(0.062,-0.112)circle (.01)--
(0.061,-0.111)circle (.01)--
(0.06,-0.109)circle (.01)--
(0.059,-0.108)circle (.01)--
(0.058,-0.107)circle (.01)--
(0.057,-0.106)circle (.01)--
(0.056,-0.105)circle (.01)--
(0.056,-0.104)circle (.01)--
(0.055,-0.103)circle (.01)--
(0.054,-0.102)circle (.01)--
(0.053,-0.101)circle (.01)--
(0.053,-0.1)circle (.01)--
(0.052,-0.099)circle (.01)--
(0.051,-0.098)circle (.01)--
(0.051,-0.097)circle (.01)--
(0.05,-0.096)circle (.01)--
(0.049,-0.095)circle (.01)--
(0.049,-0.095)circle (.01)--
(0.048,-0.094)circle (.01)--
(0.048,-0.093)circle (.01)--
(0.047,-0.092)circle (.01)--
(0.047,-0.091)circle (.01)--
(0.046,-0.091)circle (.01)--
(0.045,-0.09)circle (.01)--
(0.045,-0.089)circle (.01)--
(0.044,-0.088)circle (.01)--
(0.044,-0.088)circle (.01)--
(0.043,-0.087)circle (.01)--
(0.043,-0.086)circle (.01)--
(0.043,-0.086)circle (.01)--
(0.042,-0.085)circle (.01)--
(0.042,-0.084)circle (.01)--
(0.041,-0.084)circle (.01)--
(0.041,-0.083)circle (.01)--
(0.04,-0.082)circle (.01)--
(0.04,-0.082)circle (.01)--
(0.036,-0.076)circle (.01)--
(0.033,-0.072)circle (.01)--
(0.031,-0.068)circle (.01)--
(0.029,-0.064)circle (.01)--
(0.027,-0.061)circle (.01)--
(0.025,-0.058)circle (.01)--
(0.024,-0.056)circle (.01)--
(0.022,-0.054)circle (.01)--
(0.021,-0.052)circle (.01)--
(0.02,-0.05)circle (.01)--
(0.019,-0.048)circle (.01)--
(0.018,-0.047)circle (.01)--
(0.017,-0.045)circle (.01)--
(0.017,-0.044)circle (.01)--
(0.0169,-0.043)circle (.01)--
(0.0169,-0.042)circle (.01)--
(0.0169,-0.041)circle (.01)--
(0.0169,-0.04)circle (.01)--
(0.0169,-0.037)circle (.01)
;

\draw [color = black,->](0,0) --(1.5,0)node[right] {\small $x$};
\draw [color = black,->](0.016,-0.7) --(0.016,0.7)node[right] {\small $y$};
{\draw [color = black,fill = black](0.017,0.044)  circle (.01) node[left] {\small $w_1(\infty)$};}
{\draw [color = black,fill = black](0.017,-0.044)  circle (.01) node[left] {\small $w_2(\infty)$};}
\draw [color = black,fill = black](1,0)  circle (.01) node[above right] {\small $w(0)$};
\draw [color = black] (0.27,0.729)[below right]node{\small $\ell$};
\draw [color = black] (0.5,0.2)node{\small $\R(\Phi)$};
\end{tikzpicture}
\caption{ Points $w_1(\infty)$ and $w_2(\infty)$ are smooth exposed boundary points of $\R(\Phi)$.}
\label{RotPhi2}
\end{center}
\end{figure}

\begin{proposition}\label{extreme_pts}Let $X,f$ and $\Phi$ be as in Example \ref{ex1}. Consider $$\ell(x)=\frac{1}{130}-\frac{1}{\ln(x)},\,\text{for}\,\, x>0;\quad \ell(0)=\frac{1}{130}$$ Let $\ell_1(x)=\ell(x)$, $\ell_2(x)=-\ell(x)$, $x_1(k)=x_2(k)=e^{7-10k}$ for $k\in\bN$, $a=e^{-2}$ and $\lambda=3$. Then $w_1(\infty)=(0,\frac{1}{130})$ and $w_2(\infty)=(0,-\frac{1}{130})$ are extreme non-exposed points on $\partial\R(\Phi)$.  All points on the line segments $\left[w_2(\infty),0\right), \left(0,w_1(\infty)\right]\subset\partial\R(\Phi)$ are not rotation vectors of ground states of any direction $\alpha\in S^1$.
\end{proposition}
\begin{proof}
One may check by direct computations that the function $\ell(x)$ and the points $a,\,x_i(k)$
 satisfy all the conditions of the Proposition \ref{prop1}. Hence, $\R(\Phi)=\overline{\rm Conv}\{w(0),w_i(j): j\ge\lambda,\,i=1,2\}$ and the segment of the vertical axis between $w_1(\infty)$ and $w_2(\infty)$ is the face of $\R(\Phi)$.

 Next we show that the slopes of the lines passing through $w_1(\infty)$ and $w_1(j)$ increase without bound as $j\to\infty$. This implies that the vertical axis is the only supporting line at $w_1(\infty)$ . We compute the slope
 \begin{equation*}
   \frac{{\rm pr}_2(w_1(j)-w_1(\infty))}{{\rm pr}_1(w_1(j)-w_1(\infty))} =\frac{\sum\limits_{k=1}^{j-3}\frac{1}{10k-7}-\frac{3}{130}}{\sum\limits_{k=1}^{j-3}e^{7-10k}+3e^{-2}}
 \end{equation*}
 and see that the series in the numerator diverges whereas the series in the denominator converges. Therefore $w_1(\infty)$ is an extreme non-exposed point on the boundary of $\R(\Phi)$. By symmetry, $w_2(\infty)$ is an extreme non-exposed point as well. The statement now follows from  Theorem \ref{thm1} and Proposition \ref{symmetric}.
\end{proof}

\begin{remark}
Recall that $\mu\in \cM$ is called a maximizing measure of a potential $\varphi$ if $\int \varphi d \mu\geq \int \varphi  d\nu$ for all $\nu\in \cM$ (see \cite{Je2} for further information about maximizing measures). It follows immediately from the definitions of $\R(\Phi), H_\alpha(\Phi)$ and $F_\alpha(\Phi)$ that every invariant measure $\mu$ with $\rv(\mu)\in F_\alpha(\Phi)$ is a maximizing measure for the potential $\alpha\cdot \Phi$. Applying this observation to the example in Proposition \ref{extreme_pts} shows that for all $w\in [w_2(\infty),0)\cup(0,w_1(\infty)]$ all measures in  $\cM_\phi(w)$ are maximizing measures of the potential $\alpha\cdot \Phi$ but none of these measures is a ground state.
\end{remark}

Next, we consider the case when  points $w_1(\infty)$ and $w_2(\infty)$ are both at the origin. We show that there exists a zero temperature measure in the direction $\alpha=(-1,0)$ which is non-ergodic. Choosing the functions $\ell_1$ and $\ell_2$ in a similar way as in Proposition \ref{extreme_pts} we obtain that the vertical axis is the only supporting line at the origin.  This provides an example of a non-ergodic zero temperature measure at a smooth exposed point, which  was promised in Section \ref{exposed_pts}.

\begin{proposition}\label{non-ergodic}
Let $X,f$ and $\Phi$ be as in Example \ref{ex1}. Consider $\ell(x)=-1/\ln(x),\,\text{for}\,\, x>0;\quad \ell(0)=0$. Let $\ell_1(x)=\ell(x)$, $\ell_2(x)=-\ell(x)$, $x_1(k)=x_2(k)=e^{7-10k}$ for $k\in\bN$, $a=e^{-2}$ and $\lambda=3$. Then the origin is a smooth exposed point on $\partial\R(\Phi)$ and for $\alpha=(-1,0)$ there is a non-ergodic measure $\mu$ such that $GS(-1,0)=\{\mu\}$.
\end{proposition}
\begin{proof} As an immediate consequence of Proposition \ref{symmetric} we obtain that the origin is an exposed point of the rotation set of $\Phi$ and the vertical axis is the only supporting hyperplane at the origin.  Therefore, if $\mu\in GS(\Phi)$ and $\rv(\mu)=(0,0)$ then $\mu\in GS(\alpha)$, where $\alpha=(-1,0)$. Let $\mu_\alpha$ be such a measure. Since $\Phi$ has non-negative values and $\int \Phi\, d\mu_{\alpha}=(0,0)$, we obtain that the preimage of the origin under $\Phi$ is a set of full measure $\mu_\alpha $. On the other hand, $\Phi(\xi)=(0,0)$ if and only if either $\xi_k\in S_1$ for all $k\in\bN$ or $\xi_k\in S_2$ for all $k\in\bN$. Therefore, the support of measure $\mu_\alpha$ is contained in the union of two full shifts with alphabets $S_1$ and $S_2$. Denote the unique ergodic entropy maximizing measures for these shifts by $\mu_1$ and $\mu_2$ respectively. Since $\mu_\alpha$ also maximizes entropy at the origin, its ergodic decomposition must be a convex combination of $\mu_1$ and $\mu_2$, i.e. $\mu_\alpha=s\mu_1+(1-s)\mu_2$ for some $s\in [0,\,1]$.
Applying the operator $T$ from Proposition \ref{symmetric} we obtain
\begin{align*}
s\mu_1+(1-s)\mu_2&=\mu_\alpha\\
  &= T\circ\mu_\alpha\\
   & = sT\circ\mu_1+(1-s)T\circ\mu_2\\
  & = s\mu_2+(1-s)\mu_1
\end{align*}
and hence $s=\frac12$. It follows that $\mu_\alpha=\frac12\mu_1+\frac12\mu_2$ is the unique ground state in the direction $\alpha$, which is not ergodic.
 \end{proof}

Finally we show that  the set $\{\rv(\mu):\mu\in GS(\alpha)\}$ does not necessarily have to be a singleton.
To obtain such an example we consider
a shift map and construct an appropriate   2-dimensional potential. As a consequence of Theorem \ref{thm1} (b),
we obtain a set of ground states associated with one direction vector whose rotation vectors form a non-trivial  line segment.

\begin{theorem}\label{thm5}
Let $f:X\to X$ be the one-sided full shift  over the alphabet $\{0,1\}$. Then there exists a Lipschitz continuous potential $\Phi:\Sigma_2\to\bR^2$ and a direction vector $\alpha$ such that $\rv(GS(\alpha))$ is a non-trivial compact line segment.
\end{theorem}

\begin{proof}
We start with the definition of $\Phi=(\phi_1,\phi_2):X\to\bR^2$. Let $Y\subset \Sigma_2$ the subshift in the example of Chazottes and Hochman \cite{ChH} (see Section 3). We define $\phi_1(\xi)=dist(\xi,Y)$. It follows that $\phi_1$ is Lipschitz continuous and that $GS(\phi_1)$  contains two ergodic ground states $\mu_1$ and $\mu_2$ that are both  supported on $Y$.  Let $b>0$ such that $\R(\phi_1)=[0,b]$. Let $a>0$.  Since $\mu_1\not=\mu_2$ there exist disjoint cylinders $C_1, C_2\subset X$ with $\mu_1(C_1)>\mu_2(C_1)\geq 0$ and  $\mu_2(C_2)>\mu_1(C_2)\geq 0$.
Set $$c_1= \frac{a(\mu_1(C_2)+\mu_2(C_2))}{ \mu_1(C_1)\mu_2(C_2)-\mu_2(C_1)\mu_1(C_2)},\,\,c_2=\frac{-a (\mu_1(C_1)+\mu_2(C_1))}{\mu_1(C_1)\mu_2(C_2)-\mu_2(C_1)\mu_1(C_2)}.$$
We define  $\phi_2=  c_1  \mathbbm{1}_{C_1} +  c_2 \mathbbm{1}_{C_2}$ where $ \mathbbm{1}_C$ denotes the characteristic function of a set $C$. It follows that $\Phi=(\phi_1,\phi_2)$ is Lipschitz continuous, $\int \phi_2 d\mu_1=a$, and $\int \phi_2 d\mu_2=-a$.

We now consider the rotation set $\R(\Phi)$. Set $w_1=(0,a)$ and $w_2=(0,-a)$.  The fact that $\int \phi_1 d\mu_i=0$ for $ i=1,2$ yields $w_1,w_2\in \R(\Phi)$. Moreover, since $[0,b]=\R(\phi_1)$ there exists $w_3\in \R(\Phi)$ with ${\rm pr}_1(w_3)=b$. Hence, $\R(\Phi)$ has non-empty interior. Using that $\phi_1\geq 0$ and $w_1,w_2\in \R(\Phi)$ we conclude that the $y$-axis is a supporting hyperplane of $\R(\Phi)$.
Let $\alpha=(-1,0)$ denote the corresponding direction vector. We obtain that $\mu_1,\mu_2\in GS(\alpha)$. It now follows from Theorem \ref{thm1} (b)
that $\rv(GS(\alpha))$ is a compact line segment contained  in the $y$-axis with $w_1,w_2\in \rv(GS(\alpha))$. Therefore, the line segment with end points $w_1$ and $w_2$ is contained in $\rv(GS(\alpha))$.
\end{proof}

\begin{remark}
We note that the potential $\Phi$ in Theorem  \ref{thm5} has the feature that the curve $ t \mapsto rv(\mu_{t\alpha\cdot \Phi})$ is analytic \cite{KW1} and has infinite length.
\end{remark}


\begin{thebibliography}{99}

\bibitem{B}T. Bousch, \emph{Le poisson n'a pas d'aretes}, Annales de l'Institut Henri Poincar\'e (probabilit\'es et statistiques) \textbf{36} (2000),  489-508.

\bibitem{Br} J. Br\'emont, {\it Gibbs measures at temperature zero}, Nonlinearity {\bf 16} (2003), 419--426.

\bibitem{ChH} J. R. Chazottes and M. Hochman, {\it On the Zero-Temperature Limit of Gibbs States}, Communications in Mathematical Physics {\bf 297} (2010), no. 1, 265--281.



\bibitem{CT1}V. Climenhaga and D. Thompson, {\it Equilibrium states beyond specification and the Bowen property}, Journal of the London Mathematical Society {\bf 87} (2013), 401-427.

\bibitem{CT2}V. Climenhaga and D. Thompson, {\it Intrinsic ergodicity beyond specification: β-shifts, S-gap shifts, and their factors},
Israel Journal of Mathematics, {\bf 192} (2012), 785-817.

\bibitem{CT3}V. Climenhaga and D. Thompson,{\it Unique equilibrium states for flows and homeomorphisms with non-uniform structure}, 	arXiv:1505.03803

\bibitem{CFT}V. Climenhaga, T. Fisher and D. Thompson, {\it Unique equilibrium states for the robustly transitive diffeomorphisms of Mañé and Bonatti-Viana}, arXiv:1505.06371

\bibitem{CP}V. Climenhaga and Ya. Pesin, {\it Building thermodynamics for non-uniformly hyperbolic maps}, preprint.

\bibitem{C}G. Contreras, {\it Ground States are generically a Periodic Orbit}, arXiv:1307.0559

\bibitem{CLT}G. Contreras, A.O. Lopes and P. Thieullen,  {\it Lyapunov minimizing measures for expanding maps of the circle}, Ergodic Theory and Dynamical Systems {\bf 21} (2001), 1379–1409.



\bibitem{GM}W. Geller and M. Misiurewicz, \emph{Rotation and entropy}, Transactions of the American Mathematical Society \textbf{351} (1999), 2927-2948.

\bibitem{Gr} B. Gr\"unbaum, \emph{Convex Polytopes}, Pure and Applied Mathematics vol. XVI, Interscience,
1967.

\bibitem{K}T. Kempton,
{\it Zero temperature limits of Gibbs equilibrium states for countable Markov shifts},
Journal of Statistical Physics {\bf 143} (2011), 795–806.

\bibitem{GKLM}P. Giulietti, B. Kloeckner, A. O. Lopes and D. Marcon, {\it The calculus of thermodynamic formalism}, preprint.

\bibitem{Je} O.~Jenkinson, \emph{Rotation, entropy, and equilibrium states}, Transactions of the American Mathematical Society \textbf{353} (2001), 3713--3739.


\bibitem{Je2}O.~Jenkinson {\it Ergodic optimization}, Discrete and Continuous Dynamical Systems {\bf 15} (2006), 197-224.

\bibitem{Je4} O, Jenkinson, D. Mauldin and M. Urbanski, {\it Zero temperature limits of Gibbs-equilibrium states for countable alphabet subshifts of finite type}, Journal of Statistical Physics {\bf 119} (2005), 765--776.


\bibitem{KW1} T. Kucherenko and C. Wolf, {\it The geometry and entropy of rotation sets}, Israel Journal Mathematics {\bf 1999} (2014), 791--829.

\bibitem{KW2}T. Kucherenko and C. Wolf, \emph{Localized Pressure and equilibrium states}, Journal of Statistical Physics {\bf 160} (2015), 1529--1544.

\bibitem{KW3} T. Kucherenko and C. Wolf,  {\it Entropy and rotation sets: A toymodel approach}, Communications in Contemporary Mathematics, published online: December 11, 2015, 23 pages.



\bibitem{Le} R. Leplaideur, {\it A dynamical proof for the convergence of Gibbs measures at temperature zero}, Nonlinearity {\bf 18} (2005), 2847--2880.

\bibitem{MZ}M. Misiurewicz and K. Ziemian, \emph{Rotation sets and ergodic measures for torus homeomorphisms}, Fundamenta Mathematicae \textbf{137} (1991), 45--52.


\bibitem{Mor} I. Morris, {\it Entropy for zero-temperrature limits of Gibbs-equilibrium states for countable-alphabet subshifts of finite type}, Journal of Statistical Physics {\bf 126} (2007), 315--324.

\bibitem{vEFS}A. Van Enter, R. Fernández and A. Sokal,
{\it Regularity properties and pathologies of position-space renormalization-group transformations: scope and limitations of Gibbsian theory}, Journal of Statistical Physics {\bf 72} (1993), 879–1167.


\bibitem{Wal:81} P.~Walters, \emph{An introduction to ergodic theory},
  Graduate Texts in Mathematics 79, Springer, 1981.

  \bibitem{Z}K. Ziemian, \emph{Rotation sets for subshifts of finite type}, Fundamenta Mathematicae \textbf{146} (1995), 189-201.

\end{thebibliography}
\end{document}